\documentclass[11pt]{article}
\usepackage{pgf,tikz}
\usetikzlibrary{arrows}
\usetikzlibrary{snakes}
\usetikzlibrary{positioning}

\usepackage{dsfont}
\usepackage{graphicx,color}
\usepackage{amsmath}
\usepackage{makeidx}
\usepackage{mathrsfs}
\makeindex
\usepackage{amssymb}
\usepackage{amsthm}
\numberwithin{equation}{section}
\usepackage{algcompatible} 
\usepackage{algorithm}
\usepackage{setspace} 
\usepackage{caption}
\usepackage{subcaption}

\newtheorem{thm}{Theorem}
\newtheorem{lema}{Lemma}
\newtheorem{cor}{Corollary}					
\newtheorem{prop}{Proposition}
\theoremstyle{definition}
\newtheorem{defi}{Definition}
\newtheorem{obs}{Remark}

\newtheorem{ass}{\indent Assumption}

\def\p{^{\prime}}
\def\e{^{(\varepsilon)}}
\def\d{^{(\delta)}}

\def\Prm{^{(\delta,\ell,E,\tau)}}
\def\cI{\mathcal{I}}

\def\cB{\mathcal{B}}
\def\U{\mathrm{U}}
\def\R{\mathbb{R}}
\def\E{\mathbb{E}}
\def\Z{\mathbb{Z}}
\def\N{\mathbb{N}}
\def\P{\mathbb{P}}
\def\T{T}
\def\bU{\bar{\U}}

\def\given{\mid}
\def\one{\mathds{1}}

\def\Lamepi{\Lambda_{\varepsilon}}

\def\mproT{\mu_{\U^{(\varepsilon)}_{[0,T]}}}

\def\la{\lambda}
\def\al{\alpha}
\def\ep{\varepsilon}

\def\0{^{(0)}}

\title{\bf Hydrodynamic Limit for Spatially Structured Interacting Neurons}

\author{Aline Duarte, Guilherme Ost and Andr\'es A. Rodr\'iguez \thanks{e-mail addresses: \tt{aaariasr@unal.edu.co}, \tt{aline.duart@gmail} and \tt{guilhermeost@gmail.com}}
\vspace{.5cm} \\
{\it {Universidade de S\~ao Paulo} and GSSI - L'Aquila }
}

\date{\today}

\begin{document}

\maketitle

\begin{abstract}
We study the hydrodynamic limit of a stochastic system of neurons whose interactions are not of mean-field type and are produced by chemical and electrical synapses, and leak currents.
The system consists of $\ep^{-2}$ neurons embedded in $[0,1)^2$, each
spiking randomly according to a point process with rate depending on both its membrane potential and position. When  neuron $i$ spikes, its membrane potential
is reset to $0$ while the membrane potential of $j$ is increased by a positive value $\ep^2 a(i,j)$, if $i$ influences $j$.
Furthermore, between consecutive spikes, 
the system follows a deterministic motion due both to
electrical synapses and leak currents. 
The electrical synapses are involved in the synchronization of the membrane potentials of the  neurons, while the leak currents inhibit the activity of all neurons, attracting simultaneously their membrane potentials to 0. We show that the empirical distribution of the membrane potentials converges, as $\ep$ vanishes, to a probability density $\rho_t(u,r)$ which is proved to obey a non linear PDE of Hyperbolic type.

\end{abstract}


\section{Introduction}
In this paper we present a stochastic process which describes a population of spatially structured interacting neurons.  Our aim is to study the hydrodynamical limit of such process and characterize its limit law as well.  Despite of its own interest in mathematics, the analysis of hydrodynamical behavior of neuronal systems is an important issue in neurobiology. For instance,  the most common imaging techniques, including EEG and fMRI, do not measure individual neuron activity but rather a resulting effect driven by interactions of large subpopulations of neurons. Thus, the rigorous mathematical modeling of EEG and fMRI data requires a collective description (a ``macroscopic equation'') derived from many interacting neurons (``large microscopic systems''), a typical setting of study on hydrodynamical limits of stochastic particle systems.

In a nutshell, neurons are electrically  excitable cells whose activity consist in sudden peaks, called action potentials and often referred to as \textit{spikes}. More specifically, spikes are short-lasting electrical pulses in the membrane potential of the cell and the higher the membrane potential the higher the probability of a spike to occur.
Thus, it is quite natural to assume that the generating mechanism of spikes is given by a point process in which the spiking rate of a given neuron depends on its membrane potential.  
In this paper, we work under that assumption and additionally, assume that the membrane potential evolves under the effect of chemical and electrical synapses, and leak currents.

{\it Electrical synapses} are due to so-called gap-junction channels between neurons which induce a constant sharing of potential. The unique aspect of electrical synapses is their reciprocity. This means they are neither excitatory nor inhibitory but rather synchronizer. For each pair of neurons $(i,j)$, we modulate this synchronizing strength by $b(i,j)$, where $(i,j)\mapsto b(i,j)$ is a nonnegative symmetric function. For instance, if $N$ is the size of the set of neurons and $b(i,j)=N^{-1}$ for $i\neq j$ and $b(i,i)=0$,  the electrical synapses would push the membrane potential of each neuron to the average membrane potential of the system. In the general case, the membrane potential of each neuron is also attracted to a mean value, although this value may vary for each neuron depending on the shape of the function $b(i,j).$ 


In contrast with electrical synapses, {\it chemical synapses} are point events which can be described as follows. Each neuron $i$ with membrane potential $\U$ spikes randomly
at rate $\varphi(\U,i)$, where $\U\mapsto\varphi(\U,i)$ is a non decreasing function, positive at $\U>0$ and vanishing at $0$. This last assumption implies the absence of external stimuli.
When neuron $i$ spikes, its membrane potential is immediately reset to a resting potential $0$. Simultaneously, the neurons which are influenced by neuron $i$ receive an additional positive value to their membrane potential. Specifically, the membrane potential of neuron $j$ is increased by the value $a(i,j)$ in each spike of $i$, if the latter influences the former.  The positiveness of the function $(i,j)\mapsto a(i,j)$ means that all chemical synapses are of the excitatory type.

Additionally to the synapses, neurons loose potential to the environment along time due to leakage channels which  pushes down the membrane potential of each neuron toward the resting state. 
This constant outgoing flow of potential is referred to, in the neurobiological literature, as {\it leak currents.} For an account on these subjects we refer the reader to \cite{Gerstner:2002:SNM:583784}. 

Our model is inspired by the ones introduced in \cite{GalEva:13}, \cite{Errico:14} and \cite{AG:14}. 
For a critical reader’s guide to these papers - together with the one in \cite{Evafou:14} - we refer to \cite{EvaGalves:15}.
Our model is also an example of {\it piecewise deterministic Markov processes} introduced in 1984 by Davis in \cite{Davis:84}. Such processes combine random jump events, in our case due to the chemical synapses, with deterministic continuous evolutions, in our case due both to electrical synapses and leak currents.  The piecewise deterministic Markov processes have been used also to model neuronal systems by other authors, see for instance the papers \cite{Thieullen:12}, \cite{Errico:14}, \cite{AG:14}, \cite{Evafou:14} and  \cite{Robert:14}.

In the study of Hydrodynamic limits a {\it mean-field} type assumption is quite frequent. This means that $a(i,j)=b(i,j)=N^{-1}$ for any pair of neurons $(i,j),$  with $N$ being the size of the population of neurons.
For recent neuromathematical models adopting the mean-field assumption see, among others, the models in \cite{Errico:14} and \cite{Evafou:14}.  However, a more realistic description should incorporate the mutual distance among neurons. In order to achieve such accurate description, we use Kac potentials ideas
and techniques developed for such potentials in statistical mechanics. In our context, this means that the functions $a(i,j)$ and $b(i,j)$ considered here are quite general but are scaled by  factor $N^{-1}$, if $N$ stands for the size of the set of neurons.
For an account on hydrodynamic limits and Kac potentials we refer respectively to  \cite{ANAERRICO}-\cite{Landim} and \cite{Presutti:08}. 

To the best of our knowledge, it is the first time that stochastic modeling of spatially structured neuronal networks whose occurrences of spikes are described by Poisson processes has been addressed. Most of the mathematical models of neuronal system taking into account also spatial locations have been done with Brownian random components, see for instance \cite{touboul2014}.

For each $\varepsilon>0,$ the set of neurons is denoted by $\Lamepi=\varepsilon\Z^2\cap [0,1)^2$ and the state of our system at time $t\geq 0$
is specified  by $\U^{(\varepsilon)}(t)=\left(\U^{(\varepsilon)}_i(t), i\in \Lamepi \right)$, with $\U^{(\varepsilon)}_i(t)\in \R_{+}$. For each neuron $i\in\Lamepi$ and time $t\geq 0,$ $\U\e_i(t)$ represents the membrane potential of neuron $i$ at time $t$. Our main result, Theorem \ref{thm:2}, shows that the empirical distribution of the membrane potentials converges, as $\varepsilon\to 0$, to a law having, at each time $t$,  $\rho_t(u,r)dudr$ as a probability density. This means that, in the limit, for any set $ C \subset [0,1]^2$, interval $I \subset \mathbb R _+$ and time $t\geq 0$,
$\int_{I}\int_{C}\rho_t(u,r)dudr$ is the limit fraction of neurons located in $C$ whose membrane potentials are inside of $I$ at time $t$. This limit density  $\rho_t(u,r)$ is the unique solution of a nonlinear PDE of the hyperbolic type. 

The strategy for proving this theorem can be described in the following way. We identify the process with its empirical distribution and, as a first step, we show that the sequence of laws of the empirical distributions is tight.
Once tightness is proven, we identify the limiting law as supported by the solutions of the PDE by a coupling argument. Specifically, we first approximate the true process by a discrete space and time family of processes $Y^{(\ep,\delta,\ell, E, \tau)}$ for which the analysis of the Hydrodynamic limit is somehow easier.  
Once established the convergence to $Y\Prm$, we obtain the result by taking $\delta,\ell,E,\tau\to 0.$ A similar approach was recently used in \cite{Errico:14}, however, in the present work, we generalize their approach to the case of spatially structured interacting neurons. Finally, 
we show the solutions of the PDE are unique to get full convergence.



We organize this paper is the following way. In section \ref{sec:2}, we introduce our model and state the mains results, namely, Theorem \ref{thm:1}, \ref{thm:2} and \ref{thm:3}. In addition, at the end of the section, we argue that it is possible to work, without lost of generality, under a stronger condition on the spiking rate $\varphi$. In section \ref{sec:3}, we prove Theorem \ref{thm:1} under this stronger condition.  
In Section \ref{sec:4}, we show tightness for the sequence of laws of the empirical distributions. In section \ref{sec:5}, we define the family of  auxiliary processes as well as the coupling algorithm for the true and auxiliary processes. Moreover, we state Theorem \ref{thm:4} which claims that the auxiliary and true processes are close to each other. Its proof is postponed to Appendix \ref{proof thm 4}.
In section \ref{sec:6}, we state the hydrodynamic limit for the auxiliary process whose proof is given in the Appendix \ref{proof Thm5}.  In section \ref{sec:7}, we conclude the  proofs of Theorem \ref{thm:2} and Theorem \ref{thm:3}.  
In the Appendix \ref{generalcase}, we prove our results for general firing rates $\varphi$.

\section{Model Definition and Main Results}
\label{sec:2}

For each $\ep>0$, let $\Lamepi=\ep \Z^2\cap [0,1)^2$ be a $\ep$-mesh of the set $[0,1)^2$. The set $\Lamepi$ \ represents the set of neurons and its size is $|\Lamepi|=\ep^{-2},$ see figure ~\ref{fig1}. We consider a continuous time Markov process $(\U\e(t))_{t\geq 0}$ taking values in $\R_{+}^{\Lamepi}.$
For each  $t\geq 0$ and neuron $i\in \Lamepi$, $\U\e_i(t)$ models the membrane potential of neuron $i$ at time $t$. The global configuration at time $t\geq 0$ is denoted by 
$$\U^{(\ep)}(t)=(\U^{(\ep)}_i(t),i\in\Lamepi).$$
As usual in the theory of Markov processes, the dynamics of the processes is given through the infinitesimal generator $\mathcal{L}$. We assume that the action of $\mathcal{L}$ on any smooth test function $f:\R^{\Lamepi}_+\rightarrow \R$, is given by
\begin{equation}
\label{geradorU}
\mathcal{L}f(u)=\sum_{i \in \Lamepi} \varphi(u_i,i)[f(u+ \Delta_i(u))-f(u)]
- \sum_{i\in\Lamepi}\frac{\partial f}{\partial u_i}(u)\Big[\alpha u_i+\ep^2\sum_{j\in\Lamepi}b(i,j)(u_i-u_j)\Big],
\end{equation}
where for all $i\in\Lamepi$, the function $\Delta_i:\R^{\Lamepi}_+\rightarrow \R_+^{\Lamepi}$ is defined by
\begin{equation*}
(\Delta_i(u))_j=\left\lbrace
			\begin{array}{ll}
				\varepsilon^2 a(i,j), & \mbox{if} \ j\neq i \\
				-u_i, & \mbox{if} \ j=i
			\end{array}
	\right.,
\end{equation*}
with $a:[0,1)^2\times [0,1)^2\mapsto \R_+$ being a Lipschitz continuous function such that  $a(r,r)=0$ for all $r\in [0,1)^2$, $\alpha$ is a nonnegative parameter, $b:[0,1)^2\times [0,1)^2\mapsto \R_+$ is a symmetric Lipschitz continuous function also satisfying $b(r,r)=0$ for all $r\in [0,1)^2$, and

\begin{ass} 
\label{ass:1}
$\varphi\in C^1(\R_+\times [0,1)^2,\R_+)$ is increasing in the first variable such that for all $r\in[0,1)^2$, $\varphi(0,r)=0$.
\end{ass}
\begin{figure}[h]
\begin{center}
\begin{tikzpicture}
\draw[step=0.4cm, line width= 0.5pt] (0,0) grid (4.79,4.79);

\draw[dashed,line width=1.3pt] (0,4.8) -- (4.8,4.8);
\draw[dashed,line width=1.3pt] (4.8,0) -- (4.8,4.8);

\draw (0, -0.25) node {\textbf{(0,0)}};
\draw (0, 5.05) node {\textbf{(0,1)}};
\draw (5.05, -0.25) node {\textbf{(1,0)}};
\draw (5.05, 5.05) node {\textbf{(1,1)}};

\draw[snake=brace,line width=1pt]  (-0.1,2.4) -- (-0.1,2.8);
\draw (-0.4,2.6) node {\Large $\ep$};

\draw (2.4,5.3) node {\Large $\Lamepi$};

\foreach \x in {0, 0.4, 0.8, 1.2, 1.6, 2,2.4, 2.8, 3.2, 3.6,4,  4.4}
\foreach \y in {0, 0.4 ,0.8 ,1.2 ,1.6 ,2.0 , 2.4 , 2.8,3.2, 3.6, 4, 4.4}
{
\filldraw (\x,\y) circle (1.3pt);
}
\end{tikzpicture}
\end{center}
\caption{The $\epsilon$-mesh $\Lamepi$ of the set $[0,1)^2$.}
\label{fig1}
\end{figure}
The first term in (\ref{geradorU}) depicts how the chemical synapses are incorporated in our  model. A neuron $i$ with potential $u$ spikes at rate $\varphi(u,i)$. Intuitively this means 
that for any initial configuration $u\in \R_+^{\Lamepi}$ of the membrane potentials
$$\P(\U(t)=u+\Delta_i(u)\given \U(0)=u)=\varphi(u_i,i)t+o(t), \ \mbox{as} \ t \rightarrow 0.$$
Thus, the function $\varphi(\cdot,i)$ is called \textit{firing or spiking rate} of the neuron $i$. Notice that under such assumption neurons may have different spike rates, i.e, the function $\varphi(\cdot, i)$ may be different from $\varphi(\cdot, j).$
The function $a(\cdot, \cdot)$, appearing in the definition of $\Delta_i(\cdot)$, mimics the chemical synapses. The value $\ep^2a(i,j)$ corresponds to the energy added to the membrane potential of neuron $j$ when neuron $i$ spikes. 

The second term in (\ref{geradorU}) represents both  electrical synapses and leak currents. They describe the deterministic time evolution of the system between two consecutive spikes. More specifically, if there is no spikes in an interval of time $[a,b]$, the membrane potential of each neuron $i\in\Lamepi$ obeys the following ordinary differential equation 
\begin{equation}
\label{detmotion1}
\frac{d}{dt}\U\e_i(t)=-\alpha \U\e_i(t)-\ep^2\sum_{j\in\Lamepi}b(i,j)\left[\U\e_i(t)-\U\e_j(t)\right].
\end{equation} 
The function $b(\cdot, \cdot)$ incorporates the action of the gap-junction channels. The value $\ep^2b(i,j)$ corresponds the synchronization strength between the neurons $i$ and $j.$ Notice also that the first term of the right-hand side of (\ref{detmotion1}) pushes the membrane potential of neuron $i$ to the resting state $0$, so that we interpret $\al$ as the rate in which the membrane potential of each neuron decreases due to leak channels. 

Defining $\lambda\e_i=\ep^2\sum_{j\in\Lamepi}b(i,j)$ and $\tilde{b}(i,j)=\big(\lambda\e_i\big)^{-1} b(i,j)$,  automatically $i\mapsto\lambda\e_{i}$ and $(i,j)\mapsto \tilde{b}(i,j)$ are Lipschitz continuous functions, $\ep^2\sum_{j\in\Lamepi}\tilde{b}(i,j)=1$ and we can rewrite the ODE \eqref{detmotion1} as
\begin{equation}
\label{detrmotion2}
\frac{d}{dt}\U\e_i(t)=-\alpha \U\e_i(t)-\lambda\e_i\big[\U\e_i(t)-\bU\e_i(t)\big],
\end{equation} 
where for each $t\geq 0$ and $i\in\Lamepi$, $$\bU\e_i(t)=\ep^2\sum_{j\in\Lamepi}\tilde b(i,j)\U\e_j(t).$$ 
We call $\bU\e_i(t)$ the \textit{local average potential} of the neuron $i$ at time $t$. Thus, the second term of both ODE's is, in fact,  pushing with rate $\lambda\e_i$ the membrane potential of neuron $i$ to an average value which depends on $i$ itself. 

We shall study a simpler situation in which all the rates $\lambda\e_i$ - and consequently the function $(i,j)\mapsto  b(i,j)$ - do not change with $\ep$, keeping all others properties. In this way, hereafter we shall assume that there exist functions $\la:[0,1)^2\mapsto \R_+$ and $b:[0,1)^2\times [0,1)^2 \mapsto \R_+$ satisfying:

\begin{enumerate}
\item $\la$ is Lipschitz continuous;
\item $b$ is Lipschitz continuous such that for each $i\in\Lamepi$, $\ep^2\sum_{j\in\Lamepi}b(i,j)=1;$
\item Between consecutive spikes the membrane potential of each neuron $i\in\Lamepi$ obeys
\begin{equation}
\label{detrmotion3}
\frac{d}{dt}\U\e_t(i)=-\alpha \U\e_t(i)-\la_i(\U\e_t(i)-\bU\e_i(t)),
\end{equation} 
where for each $t\geq 0$ and $i\in\Lamepi$, $\bU\e_i(t)=\ep^2\sum_{j\in\Lamepi}b(i,j)\U\e_j(t)$.
\end{enumerate}
%
%
For each $\ep>0$, the existence and uniqueness of the solution of \eqref{detrmotion3} is simple, since it is a finite system of linear differential equations. For each $t\geq 0$, the unique solution, with value $u\in \R_+^{\Lamepi}$ at $0$, is given by $\Psi_t(u)=e^{At}u,$ where $A$ is a symmetric matrix whose entries depend on $\alpha, \ep^2$, $b$ and $\la:$
\begin{equation}
\label{solode}
A=(A_{i,j}:i,j\in\Lamepi), \ A_{i,j}=\left\lbrace
			\begin{array}{ll}
				\ep^2\la_ib(i,j), & \mbox{if} \ i\neq j \\
				-\alpha-\la_i, & \mbox{if} \ i=j
			\end{array}
	\right. .
\end{equation}
In the result below, Theorem 1, we prove the existence and uniqueness of the process describe above and provide an uniform control on the maximal membrane potential of the system. 
The proof of Theorem \ref{thm:1} is omitted here since it is analogous, modulo a small modification of the notation, to the proof of Theorem 1 given in \cite{Errico:14}.
In what follows, for any vector $u\in\R_+^{\Lamepi}$, $$||u||=\max_{i\in\Lamepi}\{u_i\}.$$ With this notation, the maximum membrane potential at time $t$ is $||\U\e(t)||$.

\begin{thm}
\label{thm:1}
Assume the function $\varphi$ satisfies the Assumption \ref{ass:1}.
\begin{enumerate}
\item Given $\varepsilon>0 $ and $u \in \R_+^{\Lamepi},$ there exists a unique strong Markov process $\U^{(\varepsilon)}(t) $ taking values in $\R_+^{\Lamepi} $ starting from $u$  whose generator  is given by \eqref{geradorU}.

\item Let $P_{u}\e$ be the probability law under which the initial condition of the process $\U^{(\varepsilon)}(t)$ is $ \U^{(\varepsilon)}(0) = u\in \R_+^{\Lamepi}.$ Then for any $R > 0$ and $T> 0 $ there exists  a constant $C>0$ such that
\begin{equation}
\label{UnifBound}
 \sup_{ u: \| u\|\le R} P\e_{u}\Big[ \sup_{t\le T }\|\U^{\varepsilon}(t) \| < C\Big]  \geq 1 - c_1 e^{ - {c_2 \varepsilon^{-2}} } ,
\end{equation}
where $c_1$ and $c_2$ are suitable positive constants. All the constants $C,c_1$ and $c_2$ do not depend on $\varepsilon$.
\end{enumerate}
\end{thm}

We now focus on the hydrodynamic limit of the process $(\U\e(t))_{t\geq 0}$. We suppose that for all $\ep>0$ the following assumption holds. 
\begin{ass}
\label{ass:2}
There exists a smooth function $\psi_0:\R_+\times [0,1)^2\mapsto \R_+$ fulfilling the conditions:
\begin{enumerate}
\item For each $r \in [0,1)^2,$ $\psi_0(\cdot,r)$ is a probability density on $\R_+$ whose support is $[0,R_0]$;
\item $\psi_0(\cdot,r) > 0 $ on $ [0, R_0) ;$
\item $\left(\U^{(\varepsilon)}_i(0)\right)_{i\in\Lamepi}$ is a sequence of independent random variables,  
$\U\e_i(0)$ being distributed according to $\psi_0(u,i)du$.
\end{enumerate}
\end{ass}
\begin{obs}
The above assumption can be weakened. Indeed, all proofs work under the assumption in which items $(i)$ and $(ii)$ are replaced by $(i')$ and $(ii')$ where
\begin{enumerate}
\item[(i')] For each $r \in [0,1)^2,$ $\psi_0(\cdot,r)$ is a probability density on $\R_+$ with compact support $[0,R_0(r)]$; $\psi_0(\cdot,r) > 0 $ on $ [0, R_0(r)) .$ 
\item[(ii')] There exits a positive parameter $R_0$ such that $$\sup_{r\in(0,1]^2}R_0(r)\leq R_0<\infty.$$
\end{enumerate}
\end{obs}
Since the state space of the process changes with $\varepsilon$, it is convenient to identify our process $(\U\e(t))_{t\geq 0}$ as an element of a suitable space which is independent of $\varepsilon.$ The identification is achieved through the map
$$\R^{\Lamepi}_+\ni \U^{(\varepsilon)}(t) \mapsto \mu\e_t:=\epsilon^2 \sum_{i\in\Lamepi} \delta_{\left(\U_i^{(\varepsilon)}(t),  i\right)}.$$  
In this way we identify our process with the element $t\mapsto \mu\e_t$ of the Skorohod space $D(\R_+,\mathcal{S}^{\p})$, where $\mathcal{S}^{\p}$ is the Schwartz space of all smooth functions $\phi:\R_+\times [0,1)^2\rightarrow \R.$ The associated element $\mu\e_t$ has the nice biological interpretation of being the empirical distribution of the membrane potential of the neurons at time $t$. 

For any fixed $T>0$, we denote the restriction of the process to $[0,T]$ by $\mu\e_{[0,T]}$ which belongs to the space $D\big([0,T],S^{\p}\big).$ We write ${\cal P}_{[0,T]}^{(\varepsilon)}$ to denote the law on $D\big([0,T],S^{\p}\big)$ of the processes $\mu\e_{[0,T]}$. Our main result shows that for any positive $T$, the 
sequence of laws ${\cal P}_{[0,T]}^{(\varepsilon)}$ converges, as $ \varepsilon \to 0 ,$ to a law ${\cal P}_{[0,T]}$ on $D\big([0,T],S^{\p}\big)$ which is supported by a deterministic trajectory $$\rho:=\left(\rho_t(u,r)dudr\right)_{t\in[0,T],u\in \R_+,r\in [0,1)^2}.$$ 
The function $\rho_t(u,r)$ is interpreted as the limit density function and is proved to solve the nonlinear PDE
\begin{equation}
 \label{PDE}
\frac{\partial \rho_t(u,r)}{\partial t }+ \frac{\partial [V(u,r,\rho_t) \rho_t(u,r)]}{\partial u }=-\varphi(u,r) \rho_t(u,r), \ \ t>0,u>0 \ \mbox{and}\  r\in[0,1)^2, 
\end{equation}
where $V(u,r,\rho_t)=-\al u-\la_r(u-\bar{u}_t(r))+p_t(r),$ 
where for each $t\geq0$ and $r\in[0,1)^2,$
\begin{equation}
\label{pq}
\bar{u}_t(r)=\int_{[0,1)^2}\int_{0}^{\infty}u
b(r,r')\rho_t(u,r)dudr\p,  \ p_t(r)=\int_{[0,1)^2}\int_{0}^{\infty}
a(r\p,r)\varphi(u,r\p)\rho_t(u,s)dudr\p
\end{equation}
are respectively 
is the limit average potential and the limit value added to the membrane potential of the neurons near to the position $r$.

The boundary conditions of \eqref{PDE} are specified by
\begin{equation}
\label{PDEbound1}
\rho_0(u,r)=v_0(u,r), \ \rho_t(0,r)=v_1(t,r),
\end{equation}
where $v_0(u,r)=\psi_0(u,r)$ is given, while $v_1(t,r)$ has to be derived together with \eqref{PDE}. From our analysis we deduce that
\begin{equation}
\label{PDEbound2}
v_1(t,r)=\frac{q_t(r)}{\la_r\bar{u}_t(r)+p_t(r)},
\end{equation}
where $q_t(r)$ is the limit spiking rate of neurons close to position $r$, i.e, 
\begin{equation}
q_t(r)=\int_{0}^{\infty}
\varphi(u,r)\rho_t(u,r)du.
\end{equation}
Since we may have $v_0(0,r)\neq v_1(0,r)$, i.e, $\psi_0(0,r)\neq \frac{q_0(r)}{ \la_r\bar{u}_0(r)+p_0(r)}$, the function $\rho_t(u,r)$ may not be continuous, so that we need a weak formulation of \eqref{PDE}.

\begin{defi}
A real-valued function $\R_+\times \R_+\times [0,1)^2\ni(t,u,r)\mapsto\rho_t(u,r)$ is said to be a weak solution of \eqref{PDE}-\eqref{PDEbound1} if 
for all smooth functions $\phi(u)$, the real-valued function $\R_+\times [0,1)^2 \ni (t,r)\mapsto \int_{0}^{\infty}\phi(u)\rho_t(u,r)du$ is continuous in $t$, differentiable in $t>0$ and
\begin{eqnarray}
\label{Weaksol}
&&{\frac{\partial}{\partial t}\int_{0}^{\infty}\phi(u)\rho_t(u,r)du -\int_{0}^{\infty}\phi\p(u)V(u,r,\rho_t)\rho_t(u,r)du- \phi(0)V(0,r,\rho_t)v_1(t,r)}\nonumber\\
\lefteqn{ \qquad =-\int_{0}^{\infty}\varphi(u,r)\phi(u)\rho_t(u,r)du,}\\ \nonumber
&& \int_{0}^{\infty}\phi(u)\rho_0(u,r)du=\int_{0}^{\infty}\phi(u)\psi_0(u,r)du,
\end{eqnarray}
where $V(u,r,\rho_t)=-u\al-\la_r(u-\bar{u}_t(r))+p_t(r),$
with $\bar{u}_t(r)$ and $p_t(r)$ as in \eqref{pq}.
\end{defi}

The solution of \eqref{Weaksol} can be computed explicitly by the method of characteristics. Characteristics are curves along which the PDE reduces to an ODE. They are defined by the equation
\begin{equation}
\label{Charac}
\frac{d x(t,r)}{dt}=V(x(t,r),r,\rho_t).
\end{equation}
The solution of \eqref{Charac} on the interval $[s,t]$, with value $u$ at $s$ is denoted by $\T_{s,t}(u,r)$, $u\in \R_+$.  Its explicit expression is given by:
\begin{equation}
\label{flowChar}
T_{s,t}(u,r)=e^{-(\alpha+\la_r)(t-s)}u+\int_s^t e^{-(\alpha+\la_r)(t-h)}[\la_r\bar{u}_h(r)+p_h(r)]dh.
\end{equation}

The statement of our main theorem is the following.

\begin{thm}
\label{thm:2}
Under assumptions \ref{ass:1} and \ref{ass:2}, for any fixed $T>0$,
\begin{equation}
{\cal P}^{(\varepsilon)}_{[0,T]} \stackrel{w}{\to} {\cal P}_{[0, T ] } \  \mbox{in} \ D \big( [0, T ]  , {\cal S}'\big) \ \mbox{as} \ \varepsilon \to 0 , 
\end{equation}
where $ {\cal P}_{ [0, T ] } $ is the law on $  D \big( [0, T ] , {\cal S}'\big) $ supported by the distribution-valued trajectory $ \omega_t $ given by
$$ \omega_t ( \phi ) = \int_{[0,1)^2}\int_0^\infty  \phi (u,r) \rho_t (u,r) dudr , \quad  t \in [0, T ],$$
for all $\phi \in \mathcal{S}.$
The function $\rho_t(u,r) $ is the unique weak solution of \eqref{PDE}-\eqref{PDEbound1} with $v_0=\psi_0$ and $v_1$ given by  \eqref{PDEbound2}. Furthermore, for any $t\geq 0$ and $r\in[0,1)^2$, $\rho_t(u,r)$ has compact support in $u$ and
\begin{equation}
\rho_t(0,r) = \frac{ q_t(r)}{ \lambda_r\bar u_t(r)+p_t(r) } \ \mbox{and} \ \int_{0}^{\infty}\rho_t(u,r)du=1.
\end{equation}
The explicit expression of the solution $\rho_t(u,r)$ for $u \geq T_{0,t}(0,r),$ is:
\begin{equation}
\label{solpde1}
\rho_t (u,r) =  \psi_0 \left(T^{-1}_{0,t}(u,r)\right)
\exp \left\{ - \int_0^t \left[ \varphi\left( T^{-1}_{s,t}(u,r),r\right) - \alpha - \lambda_r \right] ds  \right\},
\end{equation}
and for $u=T_{s,t}(0,r)$ for some $0<s\leq t$,
\begin{equation}
\label{solpde2}
\rho_t (u,r) = \frac{q_s(r)}{\lambda_r\bar u_s(r)+p_s(r)} \exp \left\{ -\int_{s}^t [\varphi\left(T_{s,h}(0,r),r\right)-\al-\la_r]dh\right\}.
\end{equation}
\end{thm}
\begin{thm}
\label{thm:3}
Assume \eqref{ass:1} and \eqref{ass:2}. If additionally for all $r\in[0,1)^2,$
$$\psi_0(0,r)=\frac{q_0(r)}{\lambda_r\bar u_0(r)+p_0(r)}, \ \mbox{where} \ q_0(r)=\int_{0}^{\infty} \varphi(u,r)\psi_0(u,r)du  $$
and  $$\bar u_0(r)=\int_{[0,1)^2}\int_{0}^{\infty} ub(r,r')\psi_0(u,r)dudr', \quad p_0(r)=\int_{[0,1)^2}\int_{0}^{\infty} a(r\p,r)\varphi(u,r)\psi_0(u,r)dudr', $$
then $\rho_t(u,r)$ is a strong solution of \eqref{PDE}-\eqref{PDEbound1} with $v_0=\psi_0$ and $v_1$ given by  \eqref{PDEbound2}.
\end{thm}

The estimate in \eqref{UnifBound} provided by Theorem \ref{thm:1} implies that with probability going to $1$ as $\ep\to 0$ all the membrane potentials are uniformly bounded in the time interval $[0,T]$. Therefore, we are allowed to change the values of the spiking rate $\varphi$ for those values of membrane potentials not reached by the system of neurons. In doing this we can suppose without lost of generality that the function $\varphi$ satisfies the following stronger condition.
\begin{ass}
\label{ass:3}
$\varphi\in C^1(\R_+\times [0,1)^2,\R_+)$ is non-decreasing, Lipschitz continuous, bounded and constant for all $ u \geq u_0 $ for some $u_0 > 0.$  We denote by $\varphi^*=\|\varphi\|_\infty$ the sup norm of $\varphi$.
\end{ass}
The argument above is given precisely at the end of the Appendix \ref{generalcase}.

\section{Boundedness of the Membrane Potentials}
\label{sec:3}
\noindent

Hereafter, we work under the Assumption \ref{ass:3}. Exploiting such assumption we are able to prove a result stronger than in  Theorem  \ref{thm:1}. Its proof is analogous to the proof of Proposition 1 in \cite{Errico:14}, so that we omit it here.
\begin{prop}
\label{thm:1bis}
Let $\varphi$ be any function satisfying the Assumption \ref{ass:3}.
\begin{enumerate}
\item Given $\varepsilon>0 $ and $u \in \R_+^{\Lamepi} $ there exists a unique strong Markov process $\U\e(t) $ taking values in $\R_+^{\Lamepi} $ starting from $u$  whose generator  is given by \eqref{geradorU}.
\item
Let $N\e(t)$ be the total number of spikes in the time interval $[0,t]$.  For any $t\geq 0$, it holds
  \begin{equation}
  \label{stochbounds}
  N\e(t) \le \tilde{N}^{(\ep)}(t) \quad \text{stochastically,}
  \end{equation}
where $\tilde{N}\e(t)$ is the total number of events in the time interval $[0,t]$ of a Poison process with rate $\ep^{-2}\varphi^*$.
\item For any given  $T> 0 $, it holds that $$\sup_ {t\le T} \|\U\e(t) \| \le \| \U\e(0)\|+ a^{*}\ep^2N(T),$$ where $a^*=||a||_{\infty}$. In particular, there exist positive constants   $c_1$ and $c_2$  such that for any $\ep>0$ and $\U\e(0)$:
\begin{equation}
\label{boundass3}
 P^{(\ep)}_{ \U\e(0)}\Big[ \sup_{t\le T }\|\U\e(t) \| \le \| \U\e(0)\|+  2a^*\varphi^* T\Big]  \geq 1 - c_1e^{ -c_2 T \ep^{-2} } .
\end{equation}
The constants $c_1$ and $c_2$ do not depend on $\ep$.
\end{enumerate}
\end{prop}

\section{Tightness of the Sequence of Laws $\mathcal{P}^{(\varepsilon)}_{[0,T]}$}
\label{sec:4}

In this section we shall prove the tightness of the sequence $\mathcal{P}^{(\varepsilon)}_{[0,T]}$ under Assumption \ref{ass:3}. This is the first step to prove the Theorem \ref{thm:2}. Although the proof of the tightness is similar to the one provided in \cite{Errico:14}, we decide to keep it here for sake of completeness.

\begin{prop}\label{prop:4}
Assume Assumption \ref{ass:3}. Assume also that $ \U^{(\varepsilon)}(0)= u^{(\varepsilon)} $ satisfies the  Assumption \ref{ass:2}. Then the sequence of laws $\mathcal{P}^{(\varepsilon)}_{[0,T]}$ of $\mu\e_{[0, T ]} $ is tight in  $ D \big( [0,T] , {\cal S}'\big) .$
\end{prop}

\begin{proof}
Indeed, for any test function $\phi \in {\cal S}$ and all $t \in [0, T ], $  we write
\begin{equation*}
\mu\e_t(\phi)=\varepsilon^2 \sum_{i\in\Lamepi} \phi ( \U\e_i (t),i ).
\end{equation*}
By Mitoma (1983), we have only to check  tightness of $\mu\e_t(\phi), t \in [0, T ]  \in D \big( [0, T]  , \R \big)$  for any fixed $ \phi \in {\cal S} .$
For that sake, we shall use a tightness criterion provided by Theorem 2.6.2 of De Masi and Presutti (1991).  
The criterion requires the existence of a positive constant  $c$ such that
 \begin{equation}
 \label{conditions}
\sup_{t \le T } E \Big[ \gamma_t^{(\varepsilon)}\Big]^2  \le c,\;\; \;\;\sup_{t \le T } \Big[ \sigma_t^{(\varepsilon)}\Big]^2  \le c,
\end{equation}
where $\gamma_t^{(\varepsilon)}$ and $\sigma_t^{(\varepsilon)}$ are respectively given by   
\begin{equation}
\gamma_t^{(\varepsilon)} = \mathcal{L} [\mu\e_t(\phi)], \  \sigma_t^{(\varepsilon)} =  \mathcal{L}[
  \mu\e_t(\phi)]^2 - 2  \mu\e_t(\phi) \mathcal{L}[\mu\e_t(\phi)],
\end{equation}
being $\mathcal{L}$ the generator given by \eqref{geradorU}.
In order to show \eqref{conditions}, we compute $\gamma_t^{(\varepsilon)}.$ By its definition, 
\begin{multline*}
\gamma_t^{(\varepsilon)}  = \varepsilon^2 \sum_j  \sum_{ i \neq j } \varphi(\U^{(\varepsilon)}_j (t),j) \left[ \phi \left( \U^{(\varepsilon)}_i (t) + \varepsilon^2 a(j,i),i \right) - \phi ( \U^{(\varepsilon)}_i (t),i)\right]\\
 + \varepsilon^2 \sum_j \varphi( \U\e_j (t),j) \left[ \phi ( 0,j) - \phi ( \U\e_j (t),j)\right]\\ 
-\alpha\varepsilon^2  \sum_{ j } \phi'  ( \U\e_j (t),j)\U\e_j (t) -\varepsilon^2  \sum_{ j } \phi'  ( \U\e_j (t),j)\lambda_j[\U\e_j(t)-\bar \U\e_j(t)].
\end{multline*}
From simple calculations we deduce, from the expression above, that
\begin{multline*}
\gamma_t^{(\varepsilon)}  = \varepsilon^4 \sum_j  \sum_{ i \neq j } \varphi(\U^{(\varepsilon)}_j (t),j)\phi'(\U\e_i(t),i)a(j,i) + \varepsilon^2 \sum_j \varphi( \U\e_j (t),j) \phi ( 0,j)\\
-\varepsilon^2 \sum_j \varphi( \U\e_j (t),j)\phi ( \U\e_j (t),j)
-\alpha\varepsilon^2  \sum_{ j } \phi'  ( \U\e_j (t),j)\U\e_j (t)\\
 -\varepsilon^2  \sum_{ j } \phi'  ( \U\e_j (t),j)\lambda_j[\U\e_j(t)-\bar \U\e_j(t)] + O (\varepsilon^2), \hspace{2cm}
\end{multline*}
with
\[
O (\varepsilon^2) :=  \varepsilon^2 \sum_j \sum_{ i \neq j } \varphi(\U\e_j (t),j) \left[ \phi \left( \U\e_i (t) + \varepsilon^2a(j,i),i \right) - \phi ( \U\e_i (t),i) - \varepsilon^2 a(j,i) \phi'(\U\e_i (t),i)\right] .
\]
Now, Assumption \ref{ass:3} implies that $\varphi$ is bounded and since  $\phi $, $\phi\p$, $\phi''$, $a$, $\lambda$ are also bounded, we have that 
 there is a  positive constant $c$ so that
\[
|\gamma_t^{(\varepsilon)}| \le c \left(1+ \varepsilon^2\sum_j \U\e_j(t)+ \varepsilon^2\sum_j \bar{\U}\e_j(t) \right)
\le c \Big( 1+ 2\sup_{t\leq T} ||\U\e(t)||\Big) .
\]
By Assumption \ref{ass:2} and Proposition \ref{thm:1bis}, it follows that for a positive constant $c$ not depending on $\varepsilon,$ 
$ \sup_{t \le T } \E \Big[ \gamma_t^{(\varepsilon)} \Big]^2  \le c.$

We now turn to the proof of \eqref{conditions} for $\sigma_t^{(\varepsilon)} $. For that sake,  we write
$\mathcal{L}= L_{\rm fire}+L_{(\alpha+\lambda)}$, where $L_{\rm fire}\phi$ and $L_{(\alpha+\lambda)}\phi$ are
given respectively by the first and second terms on the right hand side of
\eqref{geradorU}. Notice that $L_{(\alpha+\lambda)}$ acts as a ``derivative'', so that we have
\[
L_{(\alpha+\lambda)}
 [ \mu\e_t(\phi)]^2 - 2   \mu\e_t(\phi) L_{(\alpha+\lambda)} [ \mu\e_t(\phi)] = 0.
\]
The equality above is directly verified. Thus, it follows that 
$$\sigma_t^{(\varepsilon)}=L_{\rm fire}[ \mu\e_t(\phi)]^2-2 \mu\e_t(\phi) L_{\rm fire} [ \mu\e_t(\phi)].$$
Since $|2 \mu\e_t(\phi)|\leq c$ and we have already proven the bound for $L_{\rm fire}[ \mu\e_t(\phi)]$, it remains only to bound uniformly in $t\leq T$ and in $\ep$, the $L^2$-norm of $L_{\rm fire}[\mu\e_t(\phi)]^2$. By definition, 
\begin{multline*}
L_{\rm fire}[ \mu\e_t(\phi)]^2= \varepsilon^4\sum_j\sum_{i,k\neq j}\varphi(\U\e_j(t),j)\Bigg[ \phi(\U\e_i(t)+\varepsilon^2a(j,i),i) \phi(\U\e_k(t)+\varepsilon^2a(j,k),k)  \\
-\phi(\U\e_i(t),i)\phi(\U\e_j(t),j)\Bigg] 
+\varepsilon^4\sum_j \varphi(\U\e_j(t),j)[\phi^2(0,j)-\phi^2(\U\e_j(t),j)] \\
+ 2\varepsilon^4 \sum_j \sum_{i\neq j}\varphi(\U\e_j(t),j)[\phi(0,j)-\phi(\U\e_j(t),j)] [\phi(\U\e_i(t)+\varepsilon^2a(j,i),i)- \phi(\U\e_i(t),i)].
\end{multline*}
Using the same type of arguments above we can show that the $L^2$-norm of this term is bounded uniformly in $t\in [0,T]$ and in $\ep$, concluding the proof. A careful analysis in the  signs of the terms above shows that in fact $\sigma_t^{(\varepsilon)}  \to 0 $ as $ \varepsilon \to 0.$
\end{proof}

\section{The Auxiliary Process and the Coupling Algorithm}
\label{sec:5}

In this section we shall define an auxiliary process which we later shall prove that it is close to the true process as $\ep\to 0$. This uniform closeness in the limit $\ep\to 0$ is the content of the Theorem \ref{thm:4}. The proof of this result is based on a coupling algorithm designed so that neurons in both processes spike together as
often as possible. 
In the section \ref{sec:6}, we analyse the hydrodynamic limit for the auxiliary process and in section \ref{sec:7} we provide the proofs of Theorem \ref{thm:2} and \ref{thm:3}.

Throughout the section $\ep$ is kept fixed so that we omit the superscript $\ep$ from $\U\e(t)$ and all variables involving in the definition of the auxiliary process. 
Before defining the auxiliary process we shall introduce three partitions. 
\begin{defi}[{\it Partition on space}]
Let $\ell>0$ be a fixed parameter such that $\ell^{-1}$ is an integer number. We then partition the set $[0,1)^2$ into half-open squares of side length $\ell:$
\begin{equation*}
\mathcal{C}_{\ell}=\big\{C_{(m_1,m_2)}: m_1,m_2\in  \ell\Z^2\cap [0,1)^2 \big\}, \ C_{(m_1,m_2)}=[m_1,m_1+l)\times [m_2,m_2+l).
\end{equation*}
Since we shall not use the form chosen for the elements of $\mathcal{C}_{\ell}$, we take any enumeration of the set $\ell\Z^2\cap [0,1)^2$ and assume that $$\mathcal{C}_{\ell}=\{C_{m}: m=1\ldots, \ell^{-2}\}.$$  
For each square $C_m$ we denote by $i_m$ its center.
\end{defi}

\begin{defi}[{\it Partition on time}]
Let $\delta$ and $\tau$ be positive numbers such that $\delta$ is divisible by $\tau$.   We partition the interval $[0,\delta)$ into intervals of length $\tau$:
\begin{equation*}
\mathcal{J}_{\tau}=\{J_h: h=1,\ldots \delta\tau^{-1}\}, \ J_h=\big[\delta-h\tau,\delta-(h-1)\tau\big). 
\end{equation*}
\end{defi}

Let us explain the role of the partitions $\mathcal{C}_{\ell}$ and $\mathcal{J}_{\tau}$ in the definition of the auxiliary process. The auxiliary process is denoted by $Y\Prm(n\delta)$ (the parameter $E$ will appear below) and is defined at discrete times $n\delta$, $n\in\N.$ Its definition is such that neurons in the square $C_m$, having potential $\U\geq 0$, spike with a constant rate $\varphi(\U,i_m)$ in the time interval $[n\delta,(n+1)\delta)$. Thus, neurons in same the square spike according with the same spiking rate $u\mapsto\varphi(u,i_m).$ Moreover, in the same interval, all
firing events after the first one are suppressed. 
\begin{figure}[h]
\begin{center}

\begin{tikzpicture}

\draw[step=0.4cm,color=gray!70, very thin] (0,0) grid (4.79,4.79);

\draw[step=1.6cm, line width=1.3pt] (0,0) grid (4.79,4.79);

\draw[dashed,line width=1.3pt] (0,4.8) -- (4.8,4.8);
\draw[dashed,line width=1.3pt] (4.8,0) -- (4.8,4.8);

\draw (0, -0.25) node {\textbf{(0,0)}};
\draw (0, 5.05) node {\textbf{(0,1)}};
\draw (5.05, -0.25) node {\textbf{(1,0)}};
\draw (5.05, 5.05) node {\textbf{(1,1)}};

\draw[snake=brace,line width=1pt]  (3.2,-0.1) -- (1.6,-0.1);
\draw (2.4,-0.5) node {\Large $\ell=4\ep$};

\draw[snake=brace,line width=1pt]  (-0.1,2.4) -- (-0.1,2.8);
\draw (-0.4,2.6) node {\Large $\ep$};

\filldraw[color=red!80] (0.8,0.8) circle (2pt);

\filldraw[color=red!80] (0.8,2.4) circle (2pt);

\filldraw[color=red!80] (0.8,4) circle (2pt);

\filldraw[color=red!80] (2.4,0.8) circle (2pt);

\filldraw[color=red!80] (2.4,2.4) circle (2pt);

\filldraw[color=red!80] (2.4,4) circle (2pt);

\filldraw[color=red!80] (4,2.4) circle (2pt);

\draw (2.4,5.3) node {\Large $\Lamepi$};

\filldraw[color=red!80] (4,4) circle (2pt);
\filldraw[color=red!80] (4,0.8) circle (2pt);

\foreach \x in {0, 0.4, 1.2, 1.6, 2, 2.8, 3.2, 3.6, 4.4}
\foreach \y in {0, 0.4 ,0.8 ,1.2 ,1.6 ,2.0 , 2.4 , 2.8,3.2, 3.6, 4, 4.4}
{
\filldraw (\x,\y) circle (1.2pt);
}

\foreach \x in {0.8,2.4,4}
\foreach \y in {0, 0.4, 1.2 ,1.6 ,2.0 , 2.8,3.2, 3.6, 4.4}
{
\filldraw (\x,\y) circle (1.2pt);
}

\draw[->] (4.05,2.4) -- (5.5,3.35);
\draw (5.8, 3.4) node {\Large $i_m$};
\end{tikzpicture}
\label{fig2}

\end{center}
\caption{The red dots represent the centers of each half-open square $C_m$ with length $\ell$.}
\end{figure}
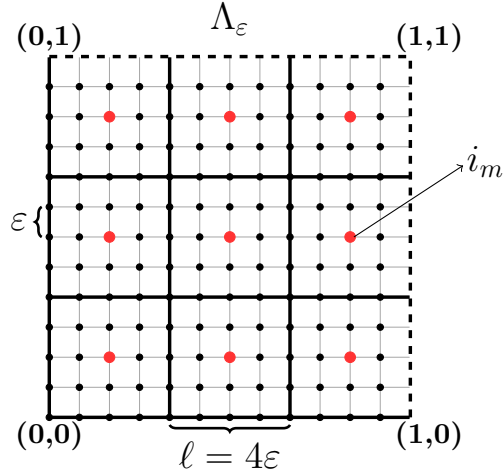 

The configuration of $Y\Prm$ is updated at every time interval $\big[n\delta,(n+1)\delta\big)$. Neurons in a common square have the same updating rule, so that we need to specify it in each square for a single neuron. For that sake, denote by $\bar{Y}\Prm_{i_m}(n\delta)$ the average potential of neuron $i_m$ in the auxiliary process at time $n\delta$ and take $i\in C_m$. 
Conditionally on $\bar{Y}\Prm_{i_m}(n\delta)=\bar{y}(i_m)$, suppose first that $i$ have not spiked during the interval $\big[n\delta,(n+1)\delta\big).$ Then the value of its membrane potential at time $(n+1)\delta$ is obtained by first letting the value of its current potential evolve, for a time $\delta$, under the attraction of $\bar{y}(i_m)$ and then taking into account the effect of the spikes in the interval $[0,\delta)$. If, on the other hand, $i$ have spiked  in the interval $J_h$, its potential is updated by first setting its current potential to $0$ and then applying the earlier updating rule during the interval $\big[\delta-(h-1)\tau ,\delta\big).$ 
This means that the potential of $i$ is then attracted for a time $(h-1)\tau$ by $\bar{y}(i_m)$ and next the effect of the spikes during $\big[\delta -(h-1)\tau ,\delta\big)$ is taken into account.
Before giving the precise definition of the auxiliary process, we need to introduce a third partition.
\begin{defi}[{\it Partition on the membrane potential at  time 0}]
Let $E$ be a positive real number which divides $R_0$. We then partition the interval $[0,R_0]$ into subintervals
\begin{equation}
\cI_{E}=\{I_{k}: k=1,\ldots, R_0E^{-1}\}, \ \ I_{k}= \big[(k-1)E,kE\big).
\end{equation} 
For each $I_k$ we denote its center by $E_{k}$.
\end{defi} 
For each neuron $i\in\Lamepi$, the value $Y\Prm_i(0)$ will be defined by first picking a point in $[0,R_0]$ according to the probability density $\psi_0(u,i)du$ and then redefining it as $E_k$ if the chosen value belongs to $I_k.$  The precise definition of the auxiliary process is given now.  

The definition of the process is done by induction. Initially, we consider the map $[0,R_0]\ni u\mapsto \Phi_0(u)$ which assigns $\Phi_0(u)=E_{k}$ if $u\in I_{k}$ and we then put 
\begin{equation}
\label{Yat0}
Y\Prm_0(i)=\Phi_0(\U_i(0)), \ \mbox{for each} \ i \in \Lamepi.
\end{equation} 
Now suppose that the configuration  $Y\Prm(n\delta)=y=(y_i,i\in\Lamepi)$ is given
and consider the sequence of independent exponential random variables $(\xi_i)_{i\in\Lamepi}$ which are independent of anything else, whose rates are $\varphi(y_{i},i_m)$ when $i\in C_{m}$. Notice that we keep constant the spiking intensity of the neurons. We write $N(m,h)$ to denote the number of neurons in $C_m$ spiking in the interval $J_{h}\in\mathcal{J}_{\tau}$,
\begin{equation}
\label{numberspikes}
N(m,h) =\sum\limits_{i\in C_m}\one_{\{\xi_i\in J_{h}\}}, \ J_h=\big[\delta-h\tau,\delta-(h-1)\tau\big),  
\end{equation}
while the contribution, due to spikes of other neurons, to the membrane potential of those neurons in $C_m$ which spike in $J_{h}$ is given by
\begin{equation}
S(m,h)= \varepsilon^2\sum\limits_{m'=1}^{\ell^{-2}}\sum\limits_{s=1}^{h-1}a(i_{m'},i_{m})N(m',s), \ h=2,\ldots, \delta\tau^{-1},
\end{equation}
and for $h=1$, we set $S(m,1)=0.$ 
%
Neurons which do not spike in $[n\delta,(n+1)\delta)$ will have their membrane potentials increased by
\begin{equation}
\label{energy0d}
S(m,\delta)= \varepsilon^2\sum\limits_{m'=1}^{\ell^{-2}}\sum\limits_{h=1}^{\delta\tau^{-1}}a(i_{m'},i_{m})N(m',h). 
\end{equation}
The average potential of neuron $i_m$ (at time $n\delta$) is defined by
\begin{equation}
\bar{y}(i_m)=\ep^2\sum_{m'=1}^{\ell^{-2}}\sum_{i\in C_{m'}}b(i_{m},i_{m'})y_i.
\end{equation}
Notice that the electrical synaptic strength is constant in each square $C_{m'}.$ 
Setting for simplicity $\bar{y}(i_m)=\bar{y}(m) $ and $\la_m=\la_{i_m}$, we write,
\begin{equation}
\label{flowy}
\Phi_{t,\bar{y}(m)}(y_i)=e^{-t(\al+\la_m)}y_i+\frac{\la_m}{\al+\la_m}\big(1-e^{-t(\al+\la_m)}\big)\bar{y}(m), \ 0\leq t\leq \delta, \ i\in C_m, \  
\end{equation}
for deterministic flow attracting the value $y_i$ to $\bar{y}_{i_m}$, and set 
\begin{equation}
\label{Ydonotspike}
Y\Prm_{i}((n+1)\delta)=\Phi_{\delta,\bar{y}(m)}(y_i)+S\big(m,\delta), \ i\in C_m, \  \mbox{if} \ \xi_i>\delta.
\end{equation}
Hence neurons which did not spike follow the deterministic flow for a time $\delta$. Afterwards, we add to their membrane potentials the value $S(m,\delta)$, generated by the spiking of other neurons, only at the end of the interval $[n\delta,(n+1)\delta).$  

For those neurons which spike in the interval $J_{h}$, we set
\begin{equation}
\label{Yspike}
Y\Prm_{i}((n+1)\delta)=
\Phi_{(h-1)\tau,\bar{y}(m)}(0)+S(m,h), \ i\in C_m, \ \mbox{if} \ \xi_i \in J_{h}.
\end{equation}
This is the value of the membrane potential of a neuron initially having potential $0$, following the deterministic flow for the remaining time $(h-1)\tau$ and receiving an additional potential $S(m,h)$, due to spikes of other neurons in the time interval $\big[\delta-(h-1)\tau,\delta\big).$

\begin{obs}
Notice that all variables $N(m,h)$, $S(m,h)$, $S(m,\delta)$ and $\bar{y}(m)$ depend on also on $n.$ We shall stress this dependency in the analysis of the hydrodynamic limits for $Y\Prm$, section \ref{sec:6}.
\end{obs}

\begin{obs}
Even though the auxiliary process $Y\Prm$ is defined in such a way that $Y\Prm$ is close to the true process, we could have chosen the distribution of the
spiking neurons in the auxiliary process differently. The choice we have made is convenient, specially in the analysis of the hydrodynamic limit for $Y\Prm$.
\end{obs}
 
\subsection{Coupling the Auxiliary and True Processes}
\label{subsec5.1}
In this section, we present a coupling algorithm for the two processes $(\U(n\delta))_{n\geq 1}$ and $(Y\Prm(n\delta))_{n\geq 1}$. The algorithm is designed so that neurons in both processes spike together as
often as possible.

At time $0$, it is set, for each $i \in \Lamepi,$ $Y\Prm_i(0)=\Phi_0(\U_i(0))$. Then, for $n\geq 0,$ the input of the algorithm is the configuration $(\U(n\delta),Y\Prm(n\delta))$ and its output is the new configuration $(\U((n+1)\delta),Y\Prm((n+1)\delta))$.
The following auxiliary variables are used in the algorithm.
\begin{itemize}
\item $(u,y)\in \R_+^{\Lamepi}\times \R_+^{\Lamepi}$ representing the configuration of membrane potentials in the two processes and $\bar{y}(m)=\ep^2\sum_{m'}\sum_{i\in C_{m'}}b(i_{m'},i_m)y_i$ representing the average membrane potential of the neuron $i_m$.
\item Independent random times $\xi^1_i, \xi^2_i, \xi_i \in (0,\infty)$, $i\in\Lamepi$, indicating possible times of updates.
\item $q=(q_i, i\in\Lamepi)\in \{0,1\}^{\Lamepi}$. The variable $q_i$ marks the possible spike of the neuron $i$ in the auxiliary process.
\item $\beta=(\beta_i,i\in\Lamepi) \in \{0, 1, \ldots\delta\tau^{-1} \}^{\Lamepi}$. The variable $\beta_i$ indicates in which subinterval of length $\tau$ the neuron $i$ has spiked in the auxiliary process. The condition $\beta_i=0$ means the neuron $i$ has not spiked.
\item $L\in[0,\delta]$ indicates the remaining time after each update of the system.
\end{itemize}
The deterministic flows follow by the processes $\U$ and $Y\Prm$ make part of the coupling algorithm. 
Recall that the deterministic flow of the process $Y_i\Prm$ is denoted by $\Phi_{t,\bar{y}(m)}(y_i)$, see equation \eqref{flowy}, while the deterministic flow of the $\U_i$ at time $t$ is $\Psi^i_{t,u}(u_i)=(e^{At}u)_i, $ see \eqref{solode} and formulas therein. 

The coupling algorithm can be described as follows. Conditionally on random vector $(\U(n\delta),Y\Prm(n\delta))=(u,y)$,  we attach to each neuron $i$ two independent random clocks $\xi^1_i$ and $\xi^2_i$. 
For $i\in C_m$, $\xi^1_i$ has intensity $\varphi(\Psi^i_{t,u}(u_i),i)\wedge\varphi(y_i,i_m)$, while $\xi^2_i$ intensity $|\varphi(\Psi^i_{t,u}(u_i),i)-\varphi(y_i,i_m)|$. Random clocks associated to different neurons are independent. 
If $\xi^1_i$ rings first, then the neuron $i$ spikes in both process and the coupling is successful. On the other hand, if $\xi^2_i$ rings first, then the neuron $i$ fires only 
in the process whose the membrane potential of $i$ at time ${\xi^2_i}_{-}$ is the largest.  Whenever the neuron $i$ fires in the interval $J_{h}$, in the auxiliary process,  we set $q_i=1$ and $\beta_i=h$ and disregard other spikes of $i$ in the auxiliary process. 
Thus, all others possible spikes of $i$ will be considered in the true process $\U_i$. For this reason we also consider a random clock $\xi_i$ with intensity $\varphi(\Psi^i_{t,u}(u_i),i)$ whose rings will indicate the next spikes of $i$ in the true process.
All the random clocks are considered only if they ring in the interval of time $[0,\delta)$.

The algorithm is provided now.  

\begin{algorithm}

\caption{Coupling algorithm}
\label{AlgCampoMedio} 
\begin{algorithmic}[1]
\State {\it Input:} $\Big(\U\e(n\delta),Y\Prm(n\delta)\Big)$
\State \textit{Output:}  $\Big(\U\e((n+1)\delta),Y\Prm((n+1)\delta)\Big)$
\State \textit{Initial values:} $(u,y)\gets \Big(\U\e(n\delta),Y\Prm(n\delta)\Big)$, $q_i\gets 0$ and $\beta_i\gets 0$, for all $i\in\Lamepi$, $L\gets \delta$ 
\WHILE{$L>0$}
	\State {For each $i \in \Lamepi$, choose independent random times 
	\begin{itemize}
	\item $\xi^1_i$ with intensity $\varphi(\Psi^i_{t,u}(u_i),i)\wedge \varphi(y_i,i_m)$ for all neurons in $C_m$
	\item $\xi^2_i$ with intensity $|\varphi(\Psi^i_{t,u}(u_i),i)-\varphi(y_i,i_m)|$ for all neurons in $C_m$
	\item $\xi_i$ with intensity $\varphi(\Psi^i_{t,u}(u_i))$
	\item $R=\inf\limits_{i\in\Lamepi;\,q_i=0}(\xi^1_i \wedge\xi^2_i)\wedge \inf\limits_{i\in\Lamepi;\,q_i=1}\xi_i$
	\end{itemize}
	}
	\IF{$R\ge L$}
	
	\textit{Stop situation:} 
	\State {$y_i\gets \Phi_{\delta,\bar{y}(m)}(y_i)+S(m,\delta)$, for all $i\in\Lamepi\cap C_m$ such that $q_i=0$}
	\State {$y_i\gets \Phi_{(\beta_i-1)\tau,\bar{y}(m)}(0)+S(m,\beta_i)$, for all $i\in\Lamepi\cap C_m$ such that $q_i=1$}
	\State {$u_i \gets \Psi^i_{L,u}(u_i)$, \ $L\gets 0$}
	\ELSIF {$R=\xi^1_i<L$}
		\State {$L\gets L-R$, \ $q_i\gets 1$, \ $\beta_i\gets \delta\tau^{-1}-\big(\big\lceil\frac{R}{\tau}\big\rceil-1\big)$ 
		\State $u_i\gets 0$,$u_j\gets \Psi^j_{R,u}(u_j)+\varepsilon^2 a(i,j)$ for all $j\ne i$}
	\ELSIF {$R=\xi^2_i<L$}
	\IF {$\varphi(\Psi^i_{R,u}(u_i),i)>\varphi(y_i,i_m)$}
		\State {$L\gets L - R$, \ $u_i\gets 0$, $u_j\gets \Psi^j_{R,u}(u)+\varepsilon^2 a(i,j)$ for all $j\ne i$} 
	\ENDIF
	\IF {$\varphi(\Psi^i_{R,u}(u_i),i)\le\varphi(y_i,i_m)$}
		\State {$L\gets L - R$, \ $q_i\gets 1$  , \ $\beta_i\gets \delta\tau^{-1}-\big(\big\lceil\frac{R}{\tau}\big\rceil-1\big),$ 
		\ $u_i\gets \Psi^i_{R,u}(u_i)$ for all $i\in\Lamepi$}
	\ENDIF
	\ELSIF {$R=\xi_i<L$}
		\State {$L\gets L - R$, \ $u_i\gets 0$, $u_j\gets \Psi^j_{R,u}(u_i)+\varepsilon^2 a(i,j)$ for all $j\ne i$}
	\ENDIF
\ENDWHILE
\State {$\big(\U\e((n+1)\delta),Y\Prm((n+1)\delta)\big)\gets (u,y)$}
\STATE \textbf{Return} $\big(\U\e((n+1)\delta),Y\Prm((n+1)\delta)\big)$
\end{algorithmic}
\end{algorithm}

\subsection{Consequences of the Coupling Algorithm}

The Theorem \ref{thm:4} is the main result of this section. It states that typically the difference of the potentials $\Delta_i(n)=|\U(n\delta)(i)-Y\Prm(n\delta)(i)|$ is small (proportionally to $\delta$). In addition, it claims that the proportion of neurons having large values of $\Delta_i(n)$ is also small (again proportional to $\delta$). 
\begin{defi}
\label{goodlabels}
A label $i \in \Lamepi$ is called ``good at time $k \delta$'' if for all $n=1 , \ldots , k$ the following is true:
\begin{itemize}
\item[(i)]
Either $\xi^1_i$ rings first and $\xi_i$ does not ring on interval $[(n-1)\delta,n\delta];$
\item[(ii)] or neither $\xi^1_i$ nor $\xi^2_i$ ring on the interval $[(n-1)\delta,n\delta].$
\end{itemize}
We denote by $\mathcal G_n$ the set of good labels at time $n\delta$ and $\mathcal B_n = \Lamepi\setminus \mathcal G_n$ the set of bad labels.
For $i \in \mathcal G_n$  we set $\Delta_i(n):= |\U(n\delta)(i)-Y\Prm(n\delta)(i)|$ so that the maximum distance between the membrane potential of the true and auxiliary process, for the good labels, is
$$ \theta_n= \max\{\Delta_i(k), \, i \in  \mathcal G_n\, , k \le n \}.$$
\end{defi}
We now enunciate the Theorem \ref{thm:4}. Its proof is postponed to the Appendix \ref{proof thm 4}.
\begin{thm}
\label{thm:4}
Grant  Assumption \ref{ass:3}, for any given $ T>0,$
there exist $ \delta_0 > 0 $ and a constant $C$ depending on $\| \varphi\|_\infty $
and on $ T$ such that for all $\delta \le \delta_0,$
             $$
\theta_n \le C \delta \quad \mbox{ and}
\quad \ep^2|\cB_n| \le C \delta \quad \mbox{ for all $n$ such that $ n \delta \le T ,$ }
            $$
with probability  $\geq 1 -c_1\delta^{-1}e^{ - c_2\epsilon^{-2} \delta^4 }.$ The constants $c_1$ and $c_2$ do not depend on $\ep$ and $\delta$.
\end{thm}
For any test function $\phi\in \mathcal{S}$ and $n\geq 1$, we write $$\nu\Prm_{n\delta}(\phi)=\ep^2\sum_{i\in\Lamepi}\phi(Y\Prm_{i}(n\delta),i_m).$$  As a by product of Theorem \ref{thm:4}, we obtain an upper bound for the $L_1$- distance between the variables $\mu_{n\delta}(\phi)$ and $\nu\Prm_{n\delta}(\phi)$, for each test function $\phi\in \mathcal{S}.$ This result will be used, in section \ref{sec:7}, in the analysis of the Hydrodynamic for $\U$. Let
\begin{equation}
\label{2.13}
\mathcal T = \Big\{ t \in [0,T]: t = n2^{-q}T, n,q \in \mathbb N\Big\} .
\end{equation}
Remember that $P^{(\ep)}_{u}$ denotes the law under which the true process $U\e(t)$ satisfies the condition $\U(0) = u.$ We write $\tilde{P}^{(\ep)}_u$ to denote the law under which the process $Y\Prm ( \cdot)$ satisfies $\Phi_0(u)=(\Phi_0(u_i),i\in \Lamepi)$ and write $Q^{(\ep)}_{u}$ to denote the joint law of the true and auxiliary processes induced by coupling algorithm provided above. We shall denote the associated expectations  by $E^{(\ep)}_u$  and $\tilde{E}^{(\ep)}_u$, and, by abuse of notation, the joint expectation by  $Q^{(\ep)}_u.$

\begin{prop}
\label{prop3}
Take $t \in \mathcal T$, $\delta\in \{2^{-q} T, q \in \mathbb N\}$ and let $n$ be such that $t=\delta n$  and fix $\phi\in \mathcal{S}$. 
Then, there exists a constant $C$, not depending on $\delta$, such that
\begin{equation}
\label{eq:445'}
Q_u^{(\ep)}\left[\big|\mu_{t}(\phi)-\nu\Prm_{t}(\phi)\big|\right]
\le C||\varphi||_{\rm Lip}\Big(\frac{e^{ - C\epsilon^{-2} \delta^4 }}{\delta} +\delta\Big).
\end{equation}
\end{prop}
The proof is given in Appendix \ref{proof prop3}. Next, in section \ref{sec:6}, we study the hydrodynamic limit for the approximating process and, in section \ref{sec:7}, we conclude the proof of Theorems \ref{thm:2} and \ref{thm:3}.

\section{Hydrodynamic Limit for the Auxiliary Process}
\label{sec:6}

In this section, we initially describe the random evolution of the membrane potentials in the auxiliary process. Next, we define a deterministic version of this evolution taking into account the average behaviour of the auxiliary process in each time interval $\big[n\delta,(n+1)\delta\big).$ 
Beside, we also consider the random variables which compute the number of neurons of the auxiliary process in a given square with a given potential and, from the dynamics of these variables, we define a second deterministic evolution. 
The main theorem of this section, Theorem \ref{thm:5}, states that both the random potentials and the counting variables becomes deterministic as $\ep\to 0$ and they are described respectively by the first and second deterministic evolutions.

In the remaining of the section these deterministic evolutions will be used to define the hydrodynamic evolution for the auxiliary processes.
When necessary
we shall stress the dependence both on $\ep$ and $n$ writing $Y^{(\ep,\delta,\ell,E,\tau)}$, $N\e_{n}(m,h)$, $\bar{y}\e_{n}(m)$, $S\e_{n}(m,h)$ and $S\e_{n}\big(m,\delta).$  
\subsection{Hydrodynamic Evolution of the Auxiliary Process}

Throughout the subsection the parameters $\delta,\ell,E,\tau$ are kept fixed, so that we omit the superscript in all variables considered below. In what follows we work in $C_m$ and doing so we drop also the dependency on $m$ unless some confusion may arise. 

We denote by $\mathcal{E}\e_n$ the random set of potentials which the auxiliary process (restricted to $C_m$) assume at time $n\delta$. By \eqref{Yat0}, we have $ \mathcal{E}\e_0=\{E\e_{0,k}:k=1\ldots, R_0E^{-1}\}$ where we set $E\e_{0,k}=E_k.$ 
At time $\delta$, 
the potential of neurons which spike in the $J_h=\big[\delta-h\tau,\delta-(h-1)\tau\big)$, independently of their initial membrane potentials, will be a value $E\e_{1,h}\in\mathcal{E}\e_1.$ By \eqref{Yspike}, we immediately see that
\begin{equation}
\label{type2}
E\e_{1,h}=\Phi_{(h-1)\tau,\bar{y}\e_{0}(m)}(0)+S\e_{1}(m,h),\ h=1,\ldots, \delta\tau^{-1}.
\end{equation}
On the other hand, at time $\delta$, the membrane potential of those neurons which initially had potential $E\e_{0,k}$ and do  not spike will be a value $E\e_{1,k+\delta\tau^{-1}}\in\mathcal{E}\e_1.$ Recalling \eqref{Ydonotspike}, it is readily verified that
\begin{equation}
\label{type1}
E^{(\ep)}_{1,k+\delta\tau^{-1}}=\Phi_{\delta,\bar{y}\e_{0}(m)}\Big(E\e_{0,k}\Big)+S\e_{1}(m,\delta), \ k=1,\ldots, |\mathcal{E}\e_0 | 
\end{equation}
where $E\e_{0,k}\in\mathcal{E}_0$.
Thus, we may split the elements of the finite set $\mathcal{E}\e_1$ into two groups. The first group consists of those potentials satisfying  \eqref{type1}, reached for only by neurons which do not spike in $[0,\delta)$. On the other hand, due to spikes of neurons in the interval  $[0,\delta)$ some potentials are ``created'' at time $\delta$. This leads to the second group of potentials, those satisfying  \eqref{type2}. Moreover, the following chain of inequalities holds 
$$0=E\e_{1,1}<\ldots <E\e_{1,\delta\tau^{-1}}<E\e_{1,1+\delta\tau^{-1}}< \ldots <E\e_{1,R_0E^{-1}+\delta\tau^{-1}}.$$
Iterating the argument above, for each $n\delta\leq T$, we may also split the elements of the set $\mathcal{E}\e_{n}$ into two groups. Those potentials belonging to the first group satisfy  
\begin{equation}
E\e_{n,k+\delta\tau^{-1}}=\Phi_{\delta,\bar{y}\e_{n-1}(m)}\Big(E\e_{n-1,k}\Big)+S\e_{n}(m,\delta), \ k=1,\ldots,  |\mathcal{E}\e_{n-1}|,
\end{equation}
where $E\e_{n-1,k}\in \mathcal{E}\e_{n-1}$,
while the  potentials of the second group satisfy
\begin{equation}
E\e_{n,h}=\Phi_{(h-1)\tau,\bar{y}\e_{n-1}(m)}(0)+S\e_{n}(m,h), \ h=1,\ldots, \delta\tau^{-1}.
\end{equation}
From our definitions, we have also that
\begin{equation*}
0=E\e_{n,1}< \ldots <E\e_{n,\delta\tau^{-1}}< E\e_{n,1+\delta\tau^{-1}}< \ldots <E\e_{n,R_0E^{-1}+n\delta\tau^{-1}}.
\end{equation*}
Now, writing 
\begin{equation}
\label{meanlocalpotential}
e\e_{n}(m)=\tilde{E}^{(\ep)}_{Y^{(\ep,\delta,\ell,E,\tau)}(0)}\big[\bar{Y}^{(\ep,\delta,\ell,E,\tau)}_{i_m}(n\delta)\big], \ 1\leq m \leq \ell^2, \ n\delta\leq T,
\end{equation}
to denote the expected value of the local average membrane potential $\bar{Y}^{(\ep,\delta,\ell,E,\tau)}_{i_m}$ at time $n\delta$, we set $\mathcal{D}\e_0=\mathcal{E}\e_0$ and then recursively define for $k=1,\ldots, |\mathcal{D}\e_{n-1}|, $
\begin{equation}
\label{energies1}
D\e_{n,k+\delta\tau^{-1}}:= \Phi_{\delta,e\e_{n-1}(m)}\big(D\e_{n-1,k}\big)+\tilde{E}_{Y^{(\ep,\delta,\ell,E,\tau)}(0)}\e\big[S\e_{n}(m,\delta)\big],\ \mbox{with} \  D\e_{n-1,k}\in \mathcal{D}\e_{n-1},
\end{equation}  
\begin{equation}
\label{energies2}
D\e_{n,h}:= \Phi_{(h-1)\tau,e\e_{n-1}(m)}(0)+\tilde{E}_{Y^{(\ep,\delta,\ell,E,\tau)}(0)}\e\Big[S\e_{n}(m,h)\Big],  \ h=1,\ldots, \delta\tau^{-1}.
\end{equation}
Given $E\e_{n,k}\in \mathcal{E}\e_n$, we write $\eta\e_{n}(m,k)$ to denote the number of neurons of $Y^{(\ep,\delta,\ell,E,\tau)}$, in $C_m$, with membrane potential $E\e_{n,k}$ at time $n\delta$. 
Finally, we write
$$\zeta\e_{0}(m,k)=\tilde{E}_{Y^{(\ep,\delta,\ell,E,\tau)}(0)}\e\Big[\eta\e_{0}(m,k)\Big],$$ 
to denote the expected number of neurons of the $Y^{(\ep,\delta,\ell,E,\tau)}$ in the square $C_m$ whose potential at time $0$ is $E\e_{0,k}$, and iteratively we set
\begin{equation}
\label{numdet1}
\zeta\e_{n}(m,k+\delta\tau^{-1})=\zeta\e_{n-1}(m,k)e^{-\delta\varphi\big(D\e_{n-1,k},i_m\big)}, \ D\e_{n-1,k}\in \mathcal{D}\e_{n-1},
\end{equation}
and for $h=1,\ldots \delta\tau^{-1}$,
\begin{equation}
\label{numdet2}
\zeta\e_{n}(m,h)= \sum\limits_{k}\zeta\e_{n-1}(m,k)\Big(e^{-(\delta-h\tau)\varphi\big(D\e_{n-1,k},i_m\big)}-e^{-(\delta-(h-1)\tau)\varphi\big(D\e_{n-1,k},i_m\big)}\Big).
\end{equation}
Suppose we have computed the number neurons of $Y^{(\ep,\delta,\ell,E,\tau)}$, in $C_m$, with a given potential $E\e_{n-1,k}$. Then, the probability of a neuron with such potential does not spike in the interval $[0,\delta)$ is exactly $e^{-\delta\varphi\big(E\e_{n-1,k},i_m\big)}.$ Thus, we expect that the number of neurons having potential $E\e_{n,k}$ at the next step satisfies
$$\eta\e_{n}(m,k+\delta\tau^{-1})\approx \eta\e_{n-1}(m,k)e^{-\delta\varphi\big(E\e_{n-1,k},i_m\big)}.$$
This relation explains \eqref{numdet1}. Similarly, we notice that the expected fraction of those neurons having potential $E\e_{n-1,k}$, which spike in the interval $J_h=\big[\delta-h\tau,\delta-(h-1)\tau\big)$ is precisely 
$$\eta_{n}(m,k)\Big(e^{-(\delta-h\tau)\varphi\big(E\e_{n-1,k},i_m\big)}-e^{-(\delta-(h-1)\tau)\varphi\big(E\e_{n-1,k},i_m\big)}\Big).$$
Then, summing over $k$ we get the random version of \eqref{numdet2}.

We shall show that the random membrane potentials $E\e_{n,k}$ are close
(proportionally to $\ep^{1/2}$) 
to the deterministic values $D\e_{n,k}$ define above. Furthermore, it will be shown that the 
collection of counting variables
$\eta\e_{n}(m,k)$ are close to the values $\zeta\e_{n}(m,k)$. Here, close  means again to be proportional to $\ep^{1/2}.$ 
\begin{thm}
\label{thm:5}
There exist positive constants $C$, $c_1$ and $c_2$, not depending on $\ep$ such that 
for all $n$ with $0\le n\delta\leq T$, $ E\e_{n,k}\in \mathcal{E}\e_{n}$  and $ D\e_{n,k}\in \mathcal{D}\e_{n},$
$$ \big| E\e_{n,k}-D\e_{n,k} \big| \le C\varepsilon^{1/2},\quad \varepsilon^2 \big|\eta\e_{n}(m,k+\delta\tau^{-1}) -\zeta\e_{n}(m,k+\delta\tau^{-1})\big| \le E\ell^2\varepsilon^{1/2}$$ 
for  $k=1,\ldots, |\mathcal{E}\e_{n}|$ and 
$$\varepsilon^2 \Big|\eta\e_{n}(m,h) -\zeta\e_{n}(m,h)\Big| \le \tau\ell^2\varepsilon^{1/2}, h=1,\ldots, \delta\tau^{-1}, $$
with probability $\geq 1-c_1 e^{c_2\varepsilon^{-1}}.$ 
\end{thm}

The proof is given in the Appendix \ref{proof Thm5}.

\begin{obs}
The constant $C,c_1$ and $c_2$ given in the Theorem \ref{thm:5}, which does depend on $\ep$, turns out to have a bad dependency on the parameters $\delta,\ell, E$ and $\tau$. However all these parameters are fixed in this section, so that  the Theorem \ref{thm:5} implies that both $E\e_{n,k}$ and $D\e_{n,k}$, as well as $\varepsilon^2\eta_{n}(m,k)$ and $ \ep^2\zeta_{n}(m,k)$ are close to each other as $\ep\to 0$ (keeping $\delta,\ell, E, \tau$ fixed).
\end{obs}

\subsection{The Limit Trajectory of the Auxiliary Process}
As a consequence of the Theorem \ref{thm:5}  we shall prove that the law of $\nu_{n\delta}$ converges in the Hydrodynamic limit to a limit law denote by $\rho^{(\delta,\ell,E,\tau)}_{n\delta}(u,r)$ to be defined below. The limit as $\ep\to 0$ of $D\e_{n,k}$ and $\zeta_{n,m}\big(D\e_{n,k}\big)$  will appear in its definition. In what follows we make explicit the dependence on $\delta,\ell,E,\tau$ writing $D^{(\ep,\delta,\ell,E,\tau)}_{n,k}$, $\zeta^{(\ep,\delta,\ell,E,\tau)}_{n}(m,k)$, $e^{(\ep,\delta,\ell,E,\tau)}_{n}(m),$ $S^{(\ep,\delta,\ell,E,\tau)}_{n}(m,\delta)$ and $S^{(\ep,\delta,\ell,E,\tau)}_{n}(m,h.)$ 

We set for each $1\leq k\leq R_0E^{-1}$ and $1\leq m\leq \ell^{-2}$,
$$
\zeta^{(\delta,\ell,E,\tau)}_{0}(m,k):=
\lim_{\varepsilon\rightarrow 0}
\varepsilon^2\zeta^{(\ep,\delta,\ell,E,\tau)}_{0}(m,k), \ I_{0,k}=I_k\in\mathcal{I}_E.
$$
By Assumption \ref{ass:2} this limit exits and  it is equal to $\int_{C_m}\int_{I_k}\psi_0(u,i_m)du.$
The value $\zeta^{(\delta,\ell,E,\tau)}_{0}(m,k)$ has the nice probabilistic meaning of being
the limit fraction of neurons, inside $C_m$,  whose membrane potential is  $D^{(\ep,\delta,\ell,E,\tau)}_{0,k}=D^{(\delta,\ell,E,\tau)}_{0,k}=E_k.$

The function $\rho\Prm_0(u,r)$ is then obtained by distributing the number $\zeta_{0,m}\big(D^{(\delta,\ell,E,\tau)}_{0,k}\big)$  uniformly over the rectangle $I_{0,k}\times C_m:$
\begin{equation}
\label{roat0}
\rho\Prm_0(u,r):=\frac{\zeta^{(\delta,\ell,E,\tau)}_{0}(m,k)}{E \ell^{2}}, \ \ (u,r)\in I_{0,k}\times C_m.
\end{equation}
We now give its definition at a general step $n\delta$. We first compute the limit potentials $D^{(\delta,\ell,E,\tau)}_{1,h}$  and $D^{(\delta,\ell,E,\tau)}_{n,k+\delta\tau^{-1}}$.  Taking the limit as $\ep\to 0$ of in the expressions \eqref{energies1} and \eqref{energies2}, it follows that 
\begin{equation}
\label{energydeter}
D^{(\delta,\ell,E,\tau)}_{n,k+\delta\tau^{-1}}= \Phi_{\delta,e^{(\delta,\ell,E,\tau)}_{n-1}(m)}\Big(D^{(\delta,\ell,E,\tau)}_{n-1,k}\Big)+s^{(\delta,\ell,E,\tau)}_{n}(m),  
\end{equation}
\begin{equation}
\label{energydeter2}
D^{(\delta,\ell,E,\tau)}_{n,h}= \Phi_{(h-1)\tau,e^{(\delta,\ell,E,\tau)}_{n-1}(m)}(0)+s^{(\delta,\ell,E,\tau,h)}_{n}(m), 
\end{equation}
where for each $n\geq 0$, the functions $e^{(\delta,\ell,E,\tau)}_{n}(m)$, $s^{(\delta,\ell,E,\tau)}_{n}(m)$ and $s^{(\delta,\ell,E,\tau,h)}_{n}(m)$ are obtained by letting $\ep\to 0:$ 
\begin{equation}
\label{barudlet}
e^{(\delta,\ell,E,\tau)}_{n}(m)=\lim_{\ep\to 0} e^{(\ep,\delta,\ell,E,\tau)}_{n}(m),
\end{equation}
\begin{equation}
\label{pdlet}
s^{(\delta,\ell,E,\tau)}_{n}(m)=\lim_{\ep\to 0} \tilde{E}_{Y^{(\ep,\delta,\ell,E,\tau)}(0)}^{(\ep)}\big[S^{(\ep,\delta,\ell,E,\tau)}_{n}(m,\delta)\big], 
\end{equation} 
\begin{equation}
\label{pdleth}
s^{(\delta,\ell,E,\tau,h)}_{n}(m)=\lim_{\ep\to 0}\tilde{E}_{Y^{(\ep,\delta,\ell,E,\tau)}(0)}^{(\ep)}\big[S^{(\ep,\delta,\ell,E,\tau)}_{n}(m,h)\big].
\end{equation}  
We need also to compute the limit as $\ep\to 0$ of the numbers $\zeta^{(\ep,\delta,\ell,E,\tau)}_{n}(m,k).$ By letting $\ep\to 0$ in \eqref{numdet1}, it is clear that 
\begin{equation}
\label{numberrelation}
\zeta^{(\delta,\ell,E,\tau)}_{n}(m,k+\delta\tau^{-1})=\zeta^{(\delta,\ell,E,\tau)}_{n-1}(m,k)e^{-\delta\varphi\Big(D^{(\delta,\ell,E,\tau)}_{n-1,k},i_m\Big)}.
\end{equation}
Similarly, sending $\ep\to 0$ in \eqref{numdet2},  we have that
\begin{equation}
\label{numberrelation2}
\zeta^{(\delta,\ell,E,\tau)}_{n}(m,h)= \sum\limits_{k}\zeta^{(\delta,\ell,E,\tau)}_{n-1}(m,k)\Big(e^{-(h-1)\tau\varphi\big(D^{(\delta,\ell,E,\tau)}_{n-1,k},i_m\big)}-e^{-h\tau\varphi\big(D^{(\delta,\ell,E,\tau)}_{n-1,k},i_m\big)}\Big).
\end{equation}

Now, consider the set of intervals $\mathcal{I}^{(\delta,\ell,E,\tau)}_{n,k}=\Big\{I^{(\delta,\ell,E,\tau)}_{n,k}\Big\}$ where for $h=1,\ldots ,\delta\tau^{-1}$, the intervals are of the form 
\begin{equation}
I^{(\delta,\ell,E,\tau)}_{n,h}=\Big[D^{(\delta,\ell,E,\tau)}_{n,h},D^{(\delta,\ell,E,\tau)}_{n,h+1}\Big),
\end{equation}
while for $k=1,\ldots, |\mathcal{D}^{(\delta,\ell,E,\tau)}_{n-1}|,$ $I^{(\delta,\ell,E,\tau)}_{n,k+\delta\tau^{-1}}$ is the interval having center in the value $D^{(\delta,\ell,E,\tau)}_{n,k+\delta\tau^{-1}}$ whose length satisfies  
\begin{equation}
|I^{(\delta,\ell,E,\tau)}_{n,k+\delta\tau^{-1}}|=e^{-(\al+\la_m)\delta}|I^{(\delta,\ell,E,\tau)}_{n-1,k+\delta\tau^{-1}}|.
\end{equation}
Finally, we set 
\begin{equation}
\label{roatnd}
\rho\Prm_{n\delta}(u,r)=\frac{\zeta^{(\delta,\ell,E,\tau)}_{n}(m,k)}{|I\Prm_{n,k}|\ell^2},  \ (u,r)\in I\Prm_{n,k}\times C_m.
\end{equation} 
Notice that $\rho\Prm_{n\delta}(u,r)$ is obtained by distributing the number $\zeta^{(\delta,\ell,E,\tau)}_{n}(m,k)$  uniformly over the rectangle $I_{n,k}\times C_m$.
Furthermore, for all $r\in[0,1)^2$, the function $\rho\Prm_{n\delta}(u,r)$ is a probability density on $\R_+$, i.e, 
\begin{equation}
1=\int_{0}^{\infty}\rho\Prm_{n\delta}(u,r)du.
\end{equation}
As an immediate consequence of the definition of $\rho\Prm_{n\delta}$ and of Theorem \ref{thm:5}, 
\begin{cor}[Hydrodynamic limit for the auxiliary process]
\label{hydro}
Let $t \in \mathcal T$, $\delta\in \{2^{-q} T, q \in \mathbb N\}$ such that $t=\delta n$ for some positive integer $n$ and $\phi\in \mathcal{S}$. 
Then almost surely, as $\ep\to 0$,
\begin{equation}
\nu\Prm_{t}(\phi)\rightarrow \int_{[0,1)^2}\int_{0}^{\infty}\phi(u,r)\rho\Prm_t(u,r)dudr+O(E+\tau+\ell).
\end{equation}
\end{cor}

\subsection{Convergence of $\rho^{(\delta,\ell,E,\tau)}_{n\delta}$ as $\ell,E,\tau\to 0$ and its Consequences}
We shall next prove that the limit evolution $\rho\Prm_{n\delta}(u,r)$ converges as $\ell,E,\tau\to 0$ to a function denoted by $\rho\d_{n\delta}(u,r).$  
Its explicit expression will be given in the Proposition \ref{convergenceofrod} below.  Before going to this proposition, we shall make some considerations which motivate the definitions of all  ingredients involved in the definition of $\rho\d_{n\delta}(u,r)$.

The convergence of $\rho\Prm_{0}(u,r)$ is direct. Indeed, by \eqref{roat0}, $\rho\Prm_{0}(u,r)=\rho^{(\ell,E)}_{0}(u,r)$ and by smoothness of $\psi_0$, defining $\rho\d_{0}(u,r)=\psi_0(u,r)$, it follows 
\begin{equation}
\label{rodelta0}
\lim_{E,\ell\to 0}  ||\rho^{(E)}_{0}-\rho\d_{0}||_{\infty}=0.
\end{equation}
Now, we set  $\bar{u}\d_0(r):=\lim_{\ell,E,\tau\to 0} e^{(\delta,\ell,E,\tau)}_{0}(m)$ and $\delta p\d_0(r):=\lim_{\ell,E,\tau\to 0}s^{(\delta,\ell,E,\tau)}_{1}(m)$ where the index $m=m(r, \ell)$ is such that  for each $\ell$ , $r\in C_m.$ Let us compute their explicit expressions. By equality \eqref{meanlocalpotential}
$$ e^{(\ep, \delta,\ell,E,\tau)}_{0}(m)=\ep^2\sum_{m'=1}^{\ell^{-2}}\sum_{k=1}^{R_0E^{-1}}b(i_{m},i_{m'}) E^{(E)}_{0,k}\zeta^{(\ep, \delta,\ell,E,\tau)}_{0}(m,k)$$
so that taking the limit as $\ep\to 0$, we get from \eqref{roat0}  that
$$ e^{(\delta,\ell,E,\tau)}_0(m)=E\ell^2\sum_{m',k}b(i_{m'},i_{m}) E^{(E)}_{0,k}\rho\Prm_0(E^{(E)}_{0,k},i_{m'}).$$
From this  last expression and using the uniform convergence in \eqref{rodelta0}, we immediately have 
$$ \bar{u}\d_{0}(r)=\int_{[0,1)^2}\int_0^{\infty}u b(r',r)\rho\d_0(u,r')dudr'.$$
We now derive the expression of $\delta p\d_0(r)$. Notice that by definition, see \eqref{energy0d}, 
$$\tilde{E}_{Y^{(\ep,\delta,\ell,E,\tau)}(0)}^{(\ep)}\big[S^{(\ep,\delta,\ell,E,\tau)}_{1}(m,\delta)\big]=\ep^2\sum_{m'=1}^{\ell^{-2}}\sum_{k=1}^{R_0E^{-1}}a(i_{m'},i_{m}) \zeta^{(\ep, \delta,\ell,E,\tau)}_{0}(m,k)\big(1-e^{-\delta\varphi(E^{(E)}_{0,k},i_{m'})}\big). $$
Thus, it follows as before that
$$s\Prm_{1}(m)=E\ell^2\sum_{m'=1}^{\ell^{-2}}\sum_{k=1}^{R_0E^{-1}}a(i_{m'},i_{m}) \rho_0\Prm\Big(E^{(E)}_{0,k},i_{m'}\Big)\big(1-e^{-\delta\varphi(E^{(E)}_{0,k},i_{m'})}\big). $$
Therefore, using again \eqref{rodelta0} and then taking $\ell,E\to 0$ in the above expression, we deduce that
$$\delta p\d_0(r):=\lim_{\ell,E,\tau\to 0}s^{(\delta,\ell,E,\tau)}_{1,m}=\int_{[0,1)^2}\int_0^{\infty} a(r',r)\rho\d_0(u,r')(1-e^{-\delta\varphi(u,r')})dudr', $$
where
$$ p\d_0(r)=\int_{[0,1)^2}\int_0^{\infty} a(r',r)\rho\d_0(u,r')\frac{(1-e^{-\delta\varphi(u,r')})}{\delta}dudr'.$$
We now shall deduce the expression of $\rho\d_{\delta}.$ Given a pair $(u,r)$, $\rho\d_{\delta}(u,r)$ is interpreted as the fraction of neurons around position $r$ having potential close of the value $u$ at time $\delta.$

Notice that by equations \eqref{numberrelation} and \eqref{roatnd}, for $(u,r)\in I\Prm_{1,k+\delta\tau^{-1}}\times C_m,$
$$\rho\Prm_{\delta}(u,r)=\rho\Prm_{0}\big(\Phi^{-1}_{\delta,e^{(\delta,\ell,E,\tau)}_{0}(m)}(u)-e^{\delta(\al+\la_m)}s^{(\delta,\ell,E,\tau)}_{1}(m),r\big)e^{-\delta[\varphi\big(E^{(E)}_{0,k},i_m\big)
-\al-\la_m]},$$
where for each $E$, $E^{(E)}_{0,k}$ is such that $\Phi^{-1}_{\delta,e^{(\delta,\ell,E,\tau)}_{0}(m)}(u)-e^{\delta(\al+\la_m)}s^{(\delta,\ell,E,\tau)}_{1}(m)\in I\Prm_{1,k}.$ From this and the above equalities, it follows that $$E^{(E)}_{0,k}\to \Phi^{-1}_{\delta,\bar{u}\d_0(r)}(u)-\delta p\d_0(r), \ \mbox{as} \ E,\ell \to 0,$$
so that 
$$\rho\d_{\delta}(u,r)=\rho\d_{0}\big(\Phi^{-1}_{\delta,\bar{u}\d_0(r)}(u)-e^{\delta(\al+\la_r)}  p\d_0(r)\delta,r\big)e^{-\delta\big[\varphi\big(\Phi^{-1}_{\delta,\bar{u}\d_0(r)}(u)-e^{\delta(\al+\la_r)}p\d_0(r)\delta,r\big)
-\al-\la_r\big]},$$
for $u\geq x_0(r)=\frac{\la_r}{\al+\la_r}(1-e^{-\delta(\al+\la_r)})\bar{u}\d_{0}(r)+\delta p_{0}\d(r).$ This formula express the flow of potentials of those neurons which do not spike in the interval $[0,\delta).$ 

Now, take a pair $(u,r)\in I\Prm_{1,h}\times C_m.$ In this case, by \eqref{energydeter2}, there exists a sequence $h=h(u,r,\tau)$ such that
\begin{equation}
\label{convergenedh}
(1-e^{-(h-1)\tau(\al+\la_m)})e\Prm_{0}(m)+s^{(\delta,\ell,E,\tau,h)}_{1}(m)=D\Prm_{1,h}\to u, \ \mbox{as} \ \tau, \ell \to 0,
\end{equation} 
and this implies that there exits a time $0\leq t_{0}(u,r)\leq \delta$ such that 
\begin{multline}
u=(1-e^{-t_0(u,r)(\al+\la_r)})\bar{u}\d_0(r)\\+\int_{[0,1)^2}\int_{0}^{\infty}a(r',r)\rho\d_{0}(v,r')\big[e^{-(\delta-t_0(v,r))\varphi(v,r')}-e^{-\delta\varphi(v,r')}\big]dvdr'.
\end{multline} 
The time in which a neuron at position $r$ has to spike in order to accumulate up to time $\delta$ a potential $u$ is exactly $\delta-t_{0}(u,r).$ 

Similarly notice that, 
$$\frac{D\Prm_{1,h+1}-D\Prm_{1,h}}{\tau} \to \la_r\bar{u}\d_0(r)e^{-(\delta-t_0(u,r))(\al+\la_r)}+\tilde{p}\d_0(r), \ \mbox{as} \ \tau,\ell\to 0,$$
where the function $\tilde{p}\d_0(r)$ is given by
$$\tilde{p}\d_0(r)=\int_{[0,1)^2}\int_{0}^{\infty} a(r',r)\varphi(u',r')\rho\d_{0}(u',r')e^{-t_0(u,r)\varphi(u',r')}du'dr'.$$
Thus, letting in $E,\ell, \tau\to 0$ in \eqref{numberrelation2} we deduce that
$$\rho\d_{\delta}(u,r)=\frac{q\d_0(r)}{\la_r\bar{u}\d_0(r)e^{-(\delta-t_0(u,r))(\al+\la_r)}+\tilde{p}\d_0(r)},$$ 
where $\bar{u}\d_0(r), \tilde{p}\d_0(r)$ and $t_0(u,r)$ as above and
$$q\d_0(r)=\int_{0}^{\infty}\rho\d_{0}(v,r)\varphi(v,r)e^{-t_h(u,r)\varphi(v,r)}dv.$$
To conclude, we notice that the convergence in \eqref{convergenedh} holds if and only if $$u<\frac{\la_r}{\al+\la_r}(1-e^{-\delta(\al+\la_r)})\bar{u}\d_{0}(r)+\delta p_{0}\d(r)=x_0(r).$$

All considerations we have done above, in fact, may be extended directly to any $n\delta$. Thus we have
\begin{prop}
\label{convergenceofrod}
For all $n\delta\leq T$, there exists the limit of $\rho\Prm_{n\delta}(u,r)$ as $\ell,E,\tau\to 0$. Moreover, let $\rho\d_{n\delta}$, $\bar{u}\d_{n\delta}(r)$ and $p\d_{n\delta}(r)$ be functions defined by 
$$\rho\d_{n\delta}(u,r)=\lim_{\ell,E,\tau\to 0}\rho\Prm_{n\delta}(u,r), \ (u,r)\in \R_+\times [0,1)^2, $$ 

\begin{equation}
\label{ubardelta}
\bar{u}\d_{n\delta}(r)=\int_{[0,1)^2}\int_{0}^{\infty}ub(r',r)\rho\d_{n\delta}(u,r')dudr',
\end{equation}
\begin{equation}
\label{pdelta}
p\d_{n\delta}(r)=\int_{[0,1)^2}\int_{0}^{\infty}a(r',r)\frac{\big(1-e^{-\delta\varphi(u,r')}\big)}{\delta}\rho\d_{n\delta}(u,r')dudr',
\end{equation}
and then set
\begin{equation}
\label{xnr}
x_n(r)=\frac{\la_r}{\al+\la_r}(1-e^{-\delta(\al+\la_r)})\bar{u}\d_{n\delta}(r)+\delta p_{n\delta}\d(r), \ r\in [0,1)^2.
\end{equation}
Then for all pairs $(u,r)$ satisfying $u\geq x_n(r),$ 
\begin{multline}
\label{solpdelta1}
\rho\d_{(n+1)\delta}(u,r)=\rho\d_{n\delta}\Big(\Phi^{-1}_{\delta,\bar{u}\d_{n\delta}(r)}(u)-e^{\delta(\al+\la_r)}  p\d_{n\delta}(r)\delta,r\Big)\\
\times \exp\Big\{-\delta\Big[\varphi\Big(\Phi^{-1}_{\delta,\bar{u}\d_{n\delta}(r)}(u)-e^{\delta(\al+\la_r)}  p\d_{n\delta}(r)\delta,r\Big)-\al-\la_r\Big]\Big\}.
\end{multline}
Now, set for any pair $(u,r)$ such that $u<x_n(r),$
\begin{equation}
\tilde{p}\d_{n\delta}(u,r)=\int_{[0,1)^2}\int_{0}^{\infty} a(r',r)\varphi(v,r')\rho\d_{n\delta}(v,r')e^{-t_{n\delta}(u,r)\varphi(v,r')}dvdr' 
\end{equation}
\begin{equation}
q\d_{n\delta}(u,r)=\int_{0}^{\infty}\rho\d_{n\delta}(v,r)\varphi(v,r)e^{-t_{n\delta}(u,r)\varphi(v,r)}dv,
\end{equation}
where the function $t_{n\delta}(u,r)$ appearing in the definition of $\tilde{p}\d_{n\delta}(u,r)$ and $q\d_{n\delta}(u,r)$ is defined through the relation
\begin{multline*}
u=\big(1-e^{-t_{n\delta}(u,r)(\al+\la_r)}\big)\bar{u}\d_{n\delta}(r)\\+\int_{[0,1)^2}\int_{0}^{\infty}a(r',r)\rho\d_{h\delta}\big[e^{-(\delta-t_{n\delta}(u,r))\varphi(v,r')}-e^{-\delta\varphi(v,r')}\big]dvdr'.
\end{multline*}
Then it follows that for all pairs $(u,r)$ with $u<x_n(r),$
\begin{equation}
\label{soldelta2}
\rho\d_{(n+1)\delta}(u,r)=\frac{q\d_{n\delta}(u,r)}{\la_r \bar{u}\d_{n\delta}(r) e^{-(\delta-t_{n\delta}(u,r))(\al+\la_r)}+\tilde{p}\d_{n\delta}(u,r)}.
\end{equation}
Furthermore, in either cases, for each $r\in[0,1)^2$ and  $n\delta\leq T$,
\begin{equation}
\label{ro.d.den}
\int_{0}^{\infty}\rho\d_{n\delta}(u,r)du=1.
\end{equation}
\end{prop}
%
Notice that by \eqref{solpdelta1}, $u\mapsto \rho\d_{(n+1)\delta}(u,r)$ has support $[0,R_{n+1}(r)]$, where
\begin{equation}
R_{n+1}(r)=e^{-\delta(\al+\la_r)}R_{n}(r)+\frac{\la_r}{\al+\la_r}(1-e^{-\delta(\al+\la_r)})\bar{u}\d_{n\delta}(r)+\delta p_{n\delta}\d(r).
\end{equation}
Since 
and $p\d_{n\delta}(r)\leq \varphi^* a^*$ and  $\bar{u}\d_{n\delta}(r)\leq R_n(r)$ it is straightforward to check that for all $n$ with $n\delta\leq T$,
\begin{equation}
R_n(r)\leq R_{n-1}(r)+\varphi^* a^* \delta\leq R_0+n\delta \varphi^*a^*\leq R_0 + T\varphi^*a^*.
\end{equation}
Thus the supports of $\rho\d_{n\delta}$ are all upper bounded by a constant which is uniform for all r and n.  An iterative application of \eqref{solpdelta1} and the explicit expression of the inverse flow $\Phi_{\delta,\bar{u}\d_{n\delta}(r)}^{ - 1 }(u) ,$ implies that
\begin{multline}
\label{3.5}
\rho^{(\delta)}_{(n+1)\delta}(u)=e^{\delta(n+1)(\al+\la)}\psi_{0}\Big(e^{  \lambda (n+1) \delta } u- \sum_{ s=0}^{n} e^{ \lambda (s+1) \delta(\al+\la_r) }\big[x_s(r)   -  2p^{(\delta)}_{s\delta}(r)\delta\big],r\Big) \\ \times 
 \exp\Big\{-\sum_{s=0}^n \delta  \varphi\Big( e^{  \lambda (n+1 - s ) \delta } u- \sum_{h=s}^n  e^{ (h-s) \delta(\al+\la_r) }\big[x_h(r)
- 2p^{(\delta)}_{h\delta}(r)\delta\Big],r\Big)\Big\}
\end{multline}
for all
\begin{equation}
\label{uestar}
u\geq u^*_{(n+1)}(r)=\sum_{s=0}^{n} e^{-\delta(n-s)(\al+\la_r)}\Big[ \frac{\la_r}{\al+\la_r}\big(1-e^{-\delta(\al+\la_r)}\big)\bar{u}\d_{s\delta}(r)+\delta p\d_{s\delta}(r)\Big],
\end{equation}
being $ \psi_0 $ the initial density and and $x_n(r)$ is defined in \eqref{xnr}.

The following results will be used in the analysis of the hydrodynamic of the true process. We first collect some properties of the functions $\bar u\d_{n\delta}(r)$ and  $p\d_{n\delta}(r).$
\begin{prop}
\label{eqconbarup}
There exist $\delta_0$ and a positive constant $C$ depending on $\varphi^*,a^*,T,R_0,\la^*$ and $\al$ such that for all $\delta \leq \delta_0$ and all $n$ satisfying $n\delta\leq T$,
\begin{enumerate}
\item $|\bar u\d_{(n+1)\delta}(r)-\bar u\d_{n\delta}(r)|+|p\d_{(n+1)\delta}(r)-\bar p\d_{n\delta}(r)|\leq C\delta,$ 
\item $|\bar u\d_{n\delta}(r)-\bar u\d_{n\delta}(r')|+|p\d_{n\delta}(r)-\bar p\d_{n\delta}(r')|\leq C|r-r'|.$
\end{enumerate}
\end{prop}
%
\begin{proof}
We will show only that
$|\bar u\d_{(n+1)\delta}(r)-\bar u\d_{n\delta}(r)|<C\delta$, since all other bounds are likewise obtained.  
By definition, 
\begin{multline}
\label{UbarraDeltaIterado}
{\bar u\d_{(n+1)\delta}(r)} = 
\int_{[0,1)^2} \int_0^{x_{n}(r\p)} ub(r\p,r)\rho\d_{(n+1)\delta}(u,r\p)du dr\p \\ + \int_{[0,1)^2}\int_{x_{n}(r\p)}^{\infty}ub(r\p,r)\rho\d_{(n+1)\delta}(u,r\p)du dr\p. \hspace{2cm}
\end{multline}
Using \eqref{solpdelta1} in the second term in \eqref{UbarraDeltaIterado} and then making change of variables
$$
v=\Phi^{-1}_{\delta,\bar{u}\d_{n\delta}(r)}(u)-\delta e^{\delta(\al+\la_r)}p\d_{n\delta}(r), \qquad \frac{d}{du}v=e^{\delta(\al+\la_r)}, 
$$
the second term in (\ref{UbarraDeltaIterado}) becomes
\begin{eqnarray*} 
\int_{[0,1)^2} \int^{\infty}_{0} b(r\p,r)\Big[\Phi_{\delta,\bar{u}\d_{n\delta}(r)}(u)+\delta p\d_{n\delta}(r)\Big]
\rho\d_{n\delta}(u,r\p)e^{-\delta\varphi(u,r\p)}dudr\p.
\end{eqnarray*}
Since  $p\d_{n\delta}(r)\leq \varphi^* a^*$ and  $\bar{u}\d_{n\delta}(r)\leq R_n(r)\leq R_0+T\varphi^*a^* $, we deduce from \eqref{UbarraDeltaIterado} and the integral above that there exists a positive constant depending on $\varphi^*,a^*, T, R_0,\la^*$ and $\alpha$ such that
\begin{equation}
\label{difubarnn1}
\big|\bar{u}\d_{(n+1)\delta}(r)-\bar{u}\d_{n\delta}(r)\big|\leq \int_{[0,1)^2} \int_0^{x_{n}(r\p)} ub(r\p,r)\rho\d_{(n+1)\delta}(u,r\p)du dr\p + C\delta.
\end{equation}
Thus, it suffices to show that the integral on the right-hand side of \eqref{difubarnn1} is $\leq C\delta.$
For that sake, we first notice that by \eqref{soldelta2}, for any pair $(u,r)$ with $u<x_{n}(r)$
\begin{equation}
\label{C1}
\rho\d_{(n+1)\delta}(u,r)\leq\frac{
\varphi^*
}{
\min\{e^{-\delta(\alpha+\lambda_{r})}, e^{-\delta\varphi^*} \}(\lambda_{r}\bar u\d_{h\delta}(r)+p\d_{h\delta}(r))
}:=C_1(\delta,n,r,\varphi^*).
\end{equation}
Then, we upper bound the integral in \eqref{difubarnn1} by
\begin{eqnarray}
\label{boundbarun1barun}
\int_{[0,1)^2}\frac{C_1(\delta,n,r',\varphi^*)}{2}x^2_{n}(r\p)b(r,r')du dr\p.  
\end{eqnarray}
Since $C_1(\delta,n,r,\varphi^*)x_{n}(r)\to 1$ as $\delta\to 0$ uniformly in $r$ and $n$, and $x_{n}(r)\leq C\delta$, we get the result from  \eqref{boundbarun1barun}.
\end{proof}
Finally, we prove equicontinuity of the function $\rho\d_{n\delta}$. The proof is an immediate consequence of the definition of $\rho\d_{n\delta}$ and the Proposition    \ref{eqconbarup}.
\begin{prop}
\label{eqconrhod}
There exists a constant C such that for all $\delta$ sufficiently small, for any $n$ and $m$, with $n\delta\leq T$, $r\in [0,1)^2$,
\begin{equation}
|\rho\d_{n\delta}(u,r)-\rho\d_{n\delta}(v,r)|\leq C\max\{|u-v|,\delta\}, \ \mbox{for} \ u,v\in \big[0,u^{*}_{n}(r)\big)
\end{equation}
and
$$|\rho\d_{n\delta}(u,r)-\rho\d_{n\delta}(v,r)|\leq C|u-v|, \ \mbox{for} \ u,v\in\big[u^{*}_{n}(r),\infty\big). $$
Moreover, for all $n\delta\leq T$ and all $r,r'\in[0,1)^2,$
\begin{equation}
|\rho\d_{n\delta}(u,r)-\rho\d_{n\delta}(u,r')|\leq C|r-r'|, \ \mbox{for} \ u\in \big[0,u^{*}_{n}(r)\wedge u^{*}_{n}(r') \big]\cup \big[u^{*}_{n}(r)\vee u^{*}_{n}(r'),\infty\big),
\end{equation}
and for all $n\delta\leq T$, $m\delta\leq T,$
\begin{equation}
|\rho\d_{n\delta}(u,r)-\rho\d_{m\delta}(u,r)|\leq C|n-m|\delta, \ \mbox{for} \ u\in \big[0,u^{*}_{n}(r)\wedge u^{*}_{m}(r) \big]\cup \big[u^{*}_{n}(r)\vee u^{*}_{m}(r),\infty\big).
\end{equation}
Furthermore, when $\psi_0$ fulfils the conditions in Theorem \eqref{thm:3}, we have additionally that
$$|\rho\d_{n\delta}{(u^{*}_{n}(r)}_+,r)-\rho\d_{n\delta}({u^{*}_{n}(r)}_-,r)|\leq C\delta $$
and
$$ |\rho\d_{n\delta}(u,r)-\rho\d_{n\delta}(u,r')|\leq C|r-r'|, \ \ |\rho\d_{n\delta}(u,r)-\rho\d_{m\delta}(u,r)|\leq C|n-m|\delta .$$
\end{prop}

\section{Hydrodynamic limit for the True Process }\label{sec:7}

We shall in the sequel complete the proof of Theorem \ref{thm:2}.  Given any positive real number $T$, recall $\mathcal T = \Big\{ t \in [0,T]: t = n2^{-q}T, q,n \in \mathbb N\Big\}.$ For each $\delta=2^{-q}T$, $q\geq 1$, we consider the following function defined on $[0,T]\times [0,1)^2:$
$$F\d_t(r)=p_{n\delta}\d(r)+\frac{\left(p_{(n+1)\delta}\d(r)-p_{n\delta}\d(r)\right)}{ \delta}(t-n\delta), \ \mbox{for} \ n\delta\leq t< (n+1)\delta, \ r\in [0,1)^2.$$
By Proposition \ref{eqconbarup} there exists a constant $C>0$ not depending on $\delta$ such that 
\begin{equation}
\label{eqcon}
|F_t\d(r)-F_s\d(r')|\leq C(|t-s|+|r-r'|), \ \mbox{for} \ (t,r)\in [0,T]\times [0,1)^2.
\end{equation}

Since \eqref{eqcon} holds for all $\delta=2^{-q}T$, it follows from the Arzel\`a-Ascoli Theorem that the sequence $F_t\d(r)$ converges by subsequences in the sup norm to a continuous function which we denote by $p\0 _t(r),t\in[0,T],r\in[0,1]^2$. In particular, it follows

\begin{equation}
\label{hyeq1}
\lim_{\delta \to 0 }\sup_{r\in[0,1]^2}\sup_{n:\ n\delta\leq T}\sup_{t\in [n\delta,(n-1)\delta]}|p_t\0(r)-p_{k\delta}\d(r)|=0.
\end{equation}
An analogous argument implies that there exists also a continuous function $\bar{u}\0 _t(r),t\in[0,T],r\in[0,1]^2$ in which
\begin{equation}
\label{hyeq2}
\lim_{\delta \to 0 }\sup_{r\in[0,1]^2}\sup_{n:\ n\delta\leq T}\sup_{t\in [n\delta,(n-1)\delta]}|\bar{u}_t\0(r)-\bar{u}_{n\delta}\d(r)|=0.
\end{equation}

Defining for each $t\in[0,T]$, $r\in[0,1]^2$:
\begin{equation}
u^{*,0}_t(r)=e^{-(\al+\la_r)t}\Bigg(\int_{0}^t\la_r\bar{u}_s\0(r)e^{(\al+\la_r)s}ds+\int_{0}^tp_s\0(r)e^{(\al+\la_r)s}ds\Bigg),
\end{equation}
it follows from \eqref{hyeq1} and \eqref{hyeq2} that
\begin{equation}
\label{hyeq3}
\lim_{\delta \to 0 }\sup_{r\in[0,1]^2}\sup_{n:\ n\delta\leq T}\sup_{t\in [n\delta,(n-1)\delta]}|u_t^{*,0}(r)-u^{*,\delta}_{n\delta}(r)|=0,
\end{equation}
where, to stress the dependence on $\delta$, we write $u^{*,\delta}_{n\delta}(r)$ instead of $u^{*}_{n\delta}(r)$ defined in \eqref{uestar}.  

In what follow we write $\zeta$ to denote the elements of the form $\zeta=2^{-q}$, with $q\in \N$. By \eqref{hyeq3}, for each $\zeta$ there exits $\delta_{\zeta}$ such that for all $\delta<\delta_{\zeta}$ we have the following. For all $r\in[0,1)^2$ and $n$ such that $n\delta\leq T$, if $|u_{n\delta}^{*,0}(r)-u|\geq \zeta$ then $u_{n\delta}^{*,0}(r)-u$ has the same sing as $u_{n\delta}^{*,\delta}(r)-u$. By using the Proposition \ref{eqconrhod} and a Arzel\`a-Ascoli type of argument to deduce that the function $\rho_t\d(u,r)$ converges uniformly to a continuous function $\rho_t(r,u),t\in\mathcal T,r\in[0,1)^2,|u-u_t^{*,0}(r)|\geq \zeta$ with compact support. We can then extend continuously $\rho_t(u,r)$ to all $t\in[0,T]$, $r\in[0,1)^2$ and $|u-u_t^{*,0}(r)|\geq \zeta$. Following a standard diagonalization procedure we the convergence above to all $t,r$ and $u$ with $u\neq u^{*,0}_t(r).$ Then by \eqref{ro.d.den},\eqref{ubardelta}, \eqref{pdelta} and the Lebesgue Dominated Convergence Theorem, for all $t\in\mathcal{T}$,

$$1=\int_{0}^{\infty}\rho_t(u,r)du, \qquad p_t\0(r)=\int_{[0,1]^2}\int_{0}^{\infty}a(r',r)\varphi(u,r')\rho_t(u,r')dudr' $$
and
$$ \bar{u}_t\0(r)=\int_{[0,1]^2}\int_{0}^{\infty}ub(r',r)\rho_t(u,r')dudr'.$$
By continuity, all these equalities hold for all $t\in[0,T].$ Hence, $p_t\0(r)$ and $\bar{u}_t\0(r)$ are equal to $p_t(r)$ and $\bar{u}_t(r)$ defined by \eqref{pq}, and therefore from now on we omit the superscript 0.
At last, by sending $\delta\to 0$ in \eqref{3.5} and \eqref{soldelta2} we show that $\rho_t(u,r)$ solves \eqref{solpde1}-\eqref{solpde2}.   

We claim that $\rho_t(u,r)$ is a weak solution for \eqref{PDE}-\eqref{PDEbound1} with $v_0=\psi_0$ and $v_1$ as in \eqref{PDEbound2}. This will be a direct consequence of the
\begin{lema}
Let $\rho_t(r,u)$ be difined as in \eqref{solpde1}-\eqref{solpde2}, then for any real valued test function $\phi$ on $\R_+\times [0,1)^2,$
\begin{multline}
\label{lemaweaksol}
\int_{[0,1)^2}\int_{0}^{\infty}\phi(u,r)\rho_t(u,r)dudr=\int_{[0,1)^2}\int_{0}^t\phi(T_{s,t}(0),r)q_s(r)\exp\Big\{-\int_{s}^t\varphi(T_{s,h}(0),r)dh\Big\}dsdr \\
+\int_{[0,1)^2}\int_{0}^{\infty}\phi(T_{0,t}(u),r)\psi_0(u,r)\exp\Big\{-\int_{0}^{t}\varphi(T_{0,s}(u),r)ds\Big\}dudr.
\end{multline}
\end{lema}
\begin{proof}
Noticing that $u^*_t(r)=T_{0,t}(0,r)$,
we start writing 
\begin{multline}
\label{proofweaksol}
\int_{[0,1)^2}\int_{0}^{\infty}\phi(u)\rho_t(u,r)dudr=\int_{[0,1)^2}\int_{0}^{u^*_t(r)}\phi(u)\rho_t(u,r)dudr
\\
+\int_{[0,1)^2}\int_{u^*_t(r)}^{\infty}\phi(u)\rho_t(u,r)dudr. \hspace{2cm}
\end{multline} 
Now, using \eqref{solpde2} and making the change of  variables $v=T_{s,t}(0)$ in the first integral of the right-hand side of \eqref{proofweaksol}, we obtain the first integral of the right-hand side of \eqref{lemaweaksol}.

To complete the proof we use \eqref{solpde1} and make the change of variable  $v=T^{-1}_{0,t}(u,r)$ in the second integral of the right-hand side of \eqref{proofweaksol}.
\end{proof}

Immediately from \eqref{lemaweaksol} follows that for every test function $\phi$, $\int_{0}^{\infty}\phi(u)\rho_t(u,r)du$ is differentiable in $t$ with its derivative fulfilling \eqref{Weaksol}. Furthermore, taking $\phi(u,r')=a(r',r)\varphi(u,r')$, $\phi(u,r')=b(r',r)u$ and $\phi(u)=\varphi(u,r)$, we conclude that the functions $\bar{u}_t(r)$, $p_t(r)$ and $q_t(r)$ are differentiable in $t$ and also that $\rho_t(u,r)$ is differentiable in $t$ and $u$ in the set $\R_+\times \R_+\times [0,1)^2\setminus \{(t,u,r):u=T_{0,t}(0,r)\}$. Thus, as a consequence of  \eqref{Weaksol} $\rho_t(u,r)$ satisfies \eqref{PDE} is this set, having also the boundary conditions \eqref{PDEbound1} $v_0=\psi_0$ and $v_1$ provided in \eqref{PDEbound2}.

We shall focus on the uniqueness for \eqref{Weaksol}. Once uniqueness is proven, we have as a by product that limit $\rho_t(u,r)$ does not depend on the converging subsequence, having in this way full convergence.
For all smooth functions $\phi$, setting $g(t,r,du)=\rho_t(u,r)du,$ we rewrite \eqref{Weaksol} in the following way,
\begin{multline}
\label{weaklaw}
\partial_t\int_{0}^{\infty}\phi(u)g(t,r,du)=\int_{0}^{\infty}\phi'(u)[-\al u - \la_r(u-\bar{u}_t(r))+p_t(r)]g(t,r,du)\\
+\int_{0}^{\infty}\varphi(u,r)[\phi(0)-\phi(u)]g(t,r,du),\hspace{1.5cm}
\end{multline}
where $g(0,r,du)=\psi_0(u,r)du$ and 
\begin{equation*}
\bar{u}_t(r)=\int\int ub(r',r)g(t,r',du)dr', \ p_t(r)=\int\int a(r',r)\varphi(u,r')g(t,r',du)dr' .
\end{equation*}
Now consider the function $\mathcal{L}_{\bar{u}_t(r),p_t(r)}$ defined on $C^1(\R_{+},\R)$, by
\begin{equation}
\label{timedepgen}
\mathcal{L}_{\bar{u}_r(t),p_t(r)}\phi(u)=\varphi(u,r)[\phi(0)-\phi(r)]+\phi'(u)[-\al u - \la_r(u-\bar{u}_t(r))+p_t(r)],
\end{equation}
and then define a stochastic process $(\U(t))_{t\geq 0}$,
$$\U(t)=\big(\U_r(t),r\in[0,1)^2\big),$$ in which, for any $r_1,\ldots, r_n\in [0,1)^2,$
\begin{enumerate}
\item the collection of stochastic processes  $(\U_{r_1}(t))_{t\geq 0},\ldots,  (\U_{r_n}(t))_{t\geq 0}$ are independent and
\item for all $r\in [0,1)^2$, the function $\mathcal{L}_{\bar{u}_t(r),p_t(r)}$ is the time-dependent generator of the process $(U_r(t))_{t\geq 0}.$
\end{enumerate}
We then deduce from equations \eqref{weaklaw} and \eqref{timedepgen} that 
for all $r$ and $t$, $g(t,r,du)$ is the law of $\U_{r}(t).$ Notice that, by independence, the law of $(\U(t))_{t\geq 0}$ is determined by the collection of laws $$\big\{g(t,r,du):t\geq 0, \ r\in[0,1)^2 \big\}.$$
On the other hand, given a stochastic processes $(\U(t))_{t\geq 0},$ $\U(t)=(\U_r(t),r\in[0,1)^2),$ fulfilling item $(i)$ above and whose time-dependent generator $\mathcal{L}_{k_t(r),h_t(r)}$ and law $p(t,r,du)$ of $U_r(t)$ are such that:
\begin{enumerate}
\item for any $r\in[0,1)^2$, $\R_{+}\ni t\mapsto p(t,r,du)$ is a continuous function; 
\item for any $t\geq 0$, $[0,1)^2\ni r\mapsto p(t,r,du)$ is a measurable function; 
\item for all $r$ and $t\leq T$, the support of $p(t,r,du)$ is contained in $[0,C_T].$
\item $k_t(r)=\int\int ub(r',r)p(t,r',du)dr', \ h_t(r)=\int\int a(r',r)\varphi(u,r')p(t,r',du)dr' $;
\item  $\mathcal{L}_{k_t(r),h_t(r)}\phi(u)=\varphi(u,r)[\phi(0)-\phi(r)]+\phi'(u)[-\al u - \la_r(u-k_t(r))+h_t(r)],\phi \in C^1(\R_{+},\R),$ 
\end{enumerate} 
then is easy to check that the law $p(t,r,du)$ satisfies \eqref{weaklaw} replacing the functions $\bar{u}_t(r)$ and $p_t(r)$ respectively by $k_t(r)$ and $h_t(r)$.

Therefore, from these considerations it follows that the  uniqueness problem of \eqref{weaklaw} reduces to prove
\begin{prop}
Let $(\U_r(t))_{t\geq 0}$ and $(V_r(t))_{t\geq 0}$ be two stochastic processes having generators $\mathcal{L}_{k_t(r),h_t(r)}$ and $\mathcal{L}_{k'_t(r),h'_t(r)}$ and laws $p(t,r,du)$ and $q(t,r,du)$ satisfying conditions $(i)-(v).$ If $\U(0)=V(0)$, then for any $T>0$, $\U_r(t)=V_r(t),0\leq t\leq T$ almost surely.
\end{prop} 
\begin{proof}
Fix $T>0$. Notice that by assumptions $U(0)=V(0)$ and $(iii)-(iv)$ there exists a constant $C>0$ such that $(k_t(r)+h_t(r))\wedge (k_t(r)+h_t(r))\leq C$, so that $t\leq T$,
\begin{equation}
\label{Ubounded}
|\U_r(t)-V_r(t)|\leq \int_{0}^{t}(\la_rk_s(r)+h_s(r))ds+ \int_{0}^{t}(\la_rk'_s(r)+h'_s(r))ds \leq 2TC. 
\end{equation}
Coupling $\U_r$ and $V_r$ such that they have the most common jumps possible, we obtain using conditions (i) and (ii) that
\begin{multline*}
\frac{d}{dt}\E\Big[|U_r(t)-V_r(t)|\Big]\leq \E\Big[|\varphi(U_r(t),r)-\varphi(V_r(t),r)|\Big(U_r(t)\wedge V_r(t)-|U_r(t)-V_r(t)|\Big)\Big]\\- \E\Big[\varphi(U_r(t),r)\wedge \varphi(V_r(t),r)|U_r(t)-V_r(t)|\Big]-(\al+\la_r)\E\Big[|U_r(t)-V_r(t)|\Big]\\
+\la_r|k_r(t)-k'_r(t)|+|h_r(t)-h'_r(t)|.\hspace{4cm}
\end{multline*}
Dropping the negative terms on the right-hand side of the inequality above, using the Lipschitz property of $\varphi$, \eqref{Ubounded}, condition $(iv)$, and writing $\gamma_r(t)=\E\Big[|U_r(t)-V_r(t)|\Big],$ we obtain for all $t\leq T$,
\begin{equation*}
\frac{d}{dt}\gamma_r(t)\leq ||\varphi||_{\rm Lip} 2TC \gamma_r(t)+\int_{[0,1)^2}\big[ \la^*b(r,r')\gamma_{r'}(t)+ \varphi^*a(r,r')\gamma_{r'}(t)\big]dr',
\end{equation*}
where $\varphi^*=||\varphi||_{\infty}$ and $\la^*=||\la||_{\infty}.$
From the inequality above, we conclude that for $t\leq T$
\begin{multline}
\label{integralform}
\gamma_r(t)\leq \tilde{C}\Big( \int_{0}^t \gamma_r(s)ds+\int_{0}^t\int_{[0,1)^2} b(r,r')\gamma_{r'}(s)dr'ds+\int_{0}^t\int_{[0,1)^2} a(r,r')\gamma_{r'}(s)dr'ds\Big),
\end{multline} 
where $\tilde{C}=\max\{||\varphi||_{\rm Lip} 2TC,\la^*,\varphi^*\}$
Iterating $n$ times the inequality in \eqref{integralform}, we get for all $t\leq T$
\begin{equation}
\gamma_r(t)\leq \tilde{C} \int_{0}^t \gamma_r(s)ds+ \frac{(\tilde{C}t)^n}{n!}\leq \tilde{C} \int_{0}^t \gamma_r(s)ds+ \frac{(\tilde{C}T)^n}{n!}
\end{equation}
Since $n$ is arbitrary, we get the the result by first letting $n\to\infty$ and then applying Gronwall's lemma. 
\end{proof}

In what follows, we shall finally prove that the true process converges to $ \rho_t (u,r)dudr$ in the hydrodynamic limit. 
The strategy is the following.
Recall
that ${\cal P}\e_{[0,T]}$ is the law on $D\big([0,T],\mathcal{S}^{\p}\big)$ of the processes $\mu\e_{[0,T]}$ and
let ${\cal P}_{[0,T]}$ be the measure valued process obtained as the limit by subsequences ${\cal P}_{[0,T]}^{(\varepsilon_i)}$.
By the tightness of ${\cal P}\e_{[0,T]}$, Proposition \ref{prop:4}, we have that this limit exits. 
Thus, the result will follow once we prove any limit measure ${\cal P}_{[0,T]}$ is supported by the deterministic trajectory $ \rho_t (u,r) dudr , t \in  [0, T], r\in [0,1)^2$ where $\rho_t(u,r) $ is the limit as $\delta\to 0$ of $ \rho_t^{(\delta ) } (u,r) $.

The following property will be used in the sequel.

\begin{prop}
\label{support}
Any weak limit $ {\cal P}_{[0,T]} $ of ${\cal P}_{[0,T]}^{(\varepsilon)}$  satisfies
$$ {\cal P}_{[0,T]} ( C ( [0, T] , {\cal S}' ) ) = 1 ,$$
where $C ( [0, T ], {\cal S}' ) $ is the space of all continuous trajectories $[0, T ] \to  {\cal S}' .$
\end{prop}
\begin{proof}
For each $\phi\in \mathcal{S}'$, consider the function  on $D([0,T],\mathcal{S}')$ given by
\begin{equation}
\Delta_{\phi}(x)=\sup_{t\in [0,T]}\big|x_t(\phi)-x_{t_{-}}(\phi)\big |.
\end{equation}
It is not difficult to prove that the function $\Delta_{\phi}$ is continuous in the Skorohod norm (see for instance De Masi and Presutti (1991), section 2.7). 
Then for any $\zeta>0$, by Chebyshev's inequality and the weak convergence 
\begin{eqnarray*}
\mathcal{P}(\{x: \Delta_{\phi}(x)>\zeta \})
\leq \zeta^{-1}\lim_{\ep\to 0}\mathcal{P}^{(\ep)}_{[0,T]}\big[\Delta_{\phi}(\mproT)\big].
\end{eqnarray*}
If there are no spikes at $t$, then it is clear that $$|\mu\e_t(\phi)-\mu\e_{t_{-}}(\phi)\big |=0.$$ 
On the other hand, if $j$ spikes at $t$, then 
\begin{multline*}
|\mu\e_t(\phi)-\mu\e_{t_{-}}(\phi)\big |=\ep^2\phi(\U\e_i(t_{-}))\\ +\ep^2\sum_{i\neq j}\big|\phi(\U\e_i(t_{-})+\ep^2a(j,i))-\phi(\U\e_i(t_{-}))\big|
\leq \ep^2\varphi^*+\ep^2\varphi^*a^*,
\end{multline*}
where in the above inequality we have used the smoothness of $\phi$. Thus, it follows that
$\lim_{\ep\to 0}\mathcal{P}^{(\ep)}_{[0,T]}\big[\Delta_{\phi}(\mu\e_{[0,T]})\big]=0,$
so that $\mathcal{P}(\{x: \Delta_{\phi}(w)>\zeta \})=0$. By the arbitrariness of $\zeta$, we deduce that $\mathcal{P}(\{x: \Delta_{\phi}(x)=0 \})=1.$ Now by the arbitrariness of $\phi,$ we conclude the proof of the theorem.
\end{proof}

We denote by $ \omega = ( \omega_t , t \in [0, T ] )$ the elements of the set $C ( [0, T ], {\cal S}' )$. Fix now $t\in {\cal T} $ and
suppose that $ {\cal P}_{[0,T]} $ is the weak limit of $ {\cal P}_{[0,T]}^{(\ep_i)} $. It will be shown now that  $ {\cal P}_{[0,T]} $ is supported by
$\{\omega:\omega_t=\rho_t(u,r)dudr\}$. Hence $ {\cal P}_{[0,T]} $ is equal to $\rho_t(u,r)dudr$ on the all rational number of $[0,T]$, so that
by continuity on all $t\in [0,T]$ which implies that weak limit  of $ {\cal P}^{(\ep)}_{[0,T]} $ is supported
by $\rho_t(u,r)dudr$.

In what follows, $ t \in {\cal T}$ and $\delta \in \{2^{-n}T, n\ge 1\}$. Fix $\zeta > 0$. Since, by Proposition \ref{support}, the support of $\mathcal{P}_{[0,T]}$ is contained in $C\big([0,T],\mathcal{S}'\big)$ and the projection is a continuous map in $C\big([0,T],\mathcal{S}'\big)$, we can use the Converging Map Theorem, see Billingsley Theorem 2.7, to deduce that 
\begin{equation}
\label{dif0}
\mathcal{P}_{[0,T]}\Big(w:|w_t(\phi)-\int\phi\rho_tdudr|>\zeta\Big)=\lim_{\ep\to 0}\mathcal{P}^{(\ep)}_{[0,T]}\big(|\mu_t^{(\ep)}(\phi)-\int\phi\rho_tdudr|>\zeta\Big).
\end{equation}
Moreover, given any $\eta>0$, 
for any $\delta$ fixed and $\ell,E,\tau$ sufficiently small we have by the Dominated Convergence Theorem that
\begin{equation}
\label{dif1}
\Big|\int\phi\rho\d_tdudr-\int\phi\rho\Prm_tdudr\Big|<\eta.
\end{equation}
For the same reason, for all $\delta$ small enough
\begin{equation}
\label{dif2}
\Big|\int\phi\rho_tdudr-\int\phi\rho\d_tdudr\Big|<\eta.
\end{equation}
Next, we fix $(\delta,\ell,E,\tau)$ such that \eqref{dif1} and \eqref{dif2} hold and then apply Corollary \eqref{hydro} for $\ep$ small enough to get that
\begin{equation*}
\label{dif3}
\Big| \tilde{E}^{(\ep)}(\nu\Prm_{t}(\phi))-\int\phi\rho\Prm_tdudr\Big|<\eta.
\end{equation*}
Furthermore, by \eqref{eq:445'} for all $\ep$ sufficiently small,
\begin{equation*}
\label{dif4}
Q_u^{(\ep)}\left[\Big|\mu\e_t(\phi)-\nu\Prm_{t}(\phi))\Big|\right]
< \eta.
\end{equation*}  
Collecting the above estimates and by
the arbitrariness of $\eta$, we then get 
\begin{equation*}
\mu\e_t(\phi)\xrightarrow{w} \int\phi\rho_tdudr \ \mbox{as} \ \ep\to 0,
\end{equation*}
and, therefore,
\begin{equation*}
\lim_{\ep\to 0}\mathcal{P}^{(\ep)}_{[0,T]}\big(|\mu\e_t(\phi)-\int\phi\rho_tdudr|>\zeta\Big)=0.
\end{equation*}
From \eqref{dif0}, it follows that $\mathcal{P}_{[0,T]}\Big(w:|w_t(\phi)-\int\phi\rho_tdudr|>\zeta\Big)=0.$ Since $\zeta$ is arbitrary  we conclude the proof of Theorem \ref{thm:2}.   

In order to complete the prove of Theorem \ref{thm:3} we have to show that
\begin{equation}
\lim_{u\nearrow u^*_t(r)}\rho_t(u,r) = \psi_0 \left(T^{-1}_{0,t}(u^*_t(r),r),r\right)
\exp \left\{ - \int_0^t \left[ \varphi\left( T^{-1}_{s,t}(u^*_t(r),r),r\right) - \alpha - \lambda(r) \right] ds  \right\},
\end{equation}
where $u^*_t(r)=T_{0,t}(0,r).$ For $u<u^*_t(r)$ we recall that,
\begin{equation}
\label{rhosol22}
\rho_t (u,r) = \frac{q_s(r)}{p_s(r)+\lambda(r)\bar u_s(r)} \exp \left\{ -\int_{s}^t \varphi\left(T_{s,h}(0,r),r\right)-\al-\la(r)dh\right\}
\end{equation}
$s$ being such that u=$T_{s,t}(0,r).$ Using the continuity
$$ \lim_{s\to 0}T_{s,h}(0,r)=T_{0,h}(0,r)=T^{-1}_{h,t}(u^*_t(r),r).$$
Since we have already shown the continuity of $q_s(r)$, $p_s(r)$ and $\bar{u}_s(r)$,
$$ \lim_{s\to 0}\frac{q_s(r)}{p_s(r)+\la_r\bar{u}_s(r)}=\frac{q_0(r)}{p_0(r)+\la_r\bar{u}_0(r)}=\psi_0(0,r)=\psi_0(T^{-1}_{0,t}(u^*_t(r),r),r).$$
Taking the limit as $s\to 0$ in \eqref{rhosol22} we finish the proof of Theorem \ref{thm:3}.
\hfill $\bullet $

%
%
%
%

\section{Acknowledges}

We are in debt to professor E. Presutti for uncountable illuminating discussions. We also thank A. De Masi for helpful discussions. 
A. Duarte and G. Ost also thank professor E. Presutti for all the teaching, attention and hospitality during their visiting in GSSI. 
This article  was produced as part of the activities of FAPESP  Research, Innovation and Dissemination Center for Neuromathematics (grant 2011/51350-6) , S.Paulo Research Foundation). A. Duarte is supported by a CNPq fellowship (grant 141270/2013-6) and G. Ost is supported by a CNPq fellowship (grant 141482/2013-3). A.A. Rodr\'iguez is supported by GSSI.

\section{Appendix 1: proof of Theorem \ref{thm:4}}
\label{proof thm 4}

The proof follows the same the steps of the proof of Theorem \ref{thm:4} of \cite{Errico:14}. We start providing an estimate of the the total number of spikes  for both processes $\U\e$ and $Y^{(\ep,\delta,\ell,E\tau)}$ in the interval $[0,T].$ Recall that $Q^{(\ep)}_u$ is the probability law governing the coupled process in which $\U\e(0)=u$ and $Y^{(\ep,\delta,\ell,E\tau)}_i(u)=\Phi_0(u_i)$ for all $i\in\Lamepi.$
\begin{prop}
\label{numDeDispEmDelta}
Let $A_{[0,T]}$ be the event when either $\U\e$ or $Y^{(\ep,\delta,\ell,E\tau)}$ have more than $2\varphi^*\epsilon^{-2}\delta$ spikes in some interval $[(k-1)\delta,k\delta)$, for $k=1,\ldots, T\delta^{-1}$. Then, under Assumption \ref{ass:3},
$$
Q_u^{(\ep)}\Big(A_{[0,T]}\Big)\le 2T\delta^{-1}e^{-\varphi^*\delta\epsilon^{-2}(3-e)},
$$
for any initial configuration $u\in \R_+^{\Lamepi}.$
\end{prop}
\begin{proof}
Fix $k\in \{1,\ldots, T\delta^{-1}\}$ and let $N\big([(k-1)\delta,k\delta)\big)$ denote the number of spikes of the $\U\e$ process in the interval $[(k-1)\delta,k\delta).$ Then, under Assumption \ref{ass:3}, $N\big([(k-1)\delta,k\delta)\big)$ is stochastically bounded by 
$$
Z:=\sum_{j\in\Lamepi} N^*_j\big([(k-1)\delta,k\delta)\big)
$$
where $(N^*_j)_{j\in\Lamepi}$ are iid Poisson processes with intensity $\varphi^*$. Since $Z$ is distributed as a Poisson random variable with rate $\ep^{-2}\delta\varphi^*$,  it follows that 
$$
Q_u^{(\ep)}(N\big([(k-1)\delta,k\delta)\big)\geq 2\varphi^*\delta\epsilon^{-2} )\leq \P(Z\ge 2\varphi^*\delta\epsilon^{-2})\le e^{-2\varphi^*\delta\epsilon^{-2}(3-e)}.
$$
Bounding in the same manner the number of spikes of the $Y\d$ process in the interval $[(k-1)\delta,k\delta)$ and then summing over $k$ we complete the proof.
\end{proof}

From now on, we suppose that, in both processes $\U\e$ and $Y^{(\ep,\delta,\ell,E\tau)}$,
the spiking rate of each neuron is $\le \varphi^*$ and the number of spikes of all neurons
in any step $[(k-1)\delta,k\delta]$ is $\le 2\varphi^*\delta \ep^{-2}$. Moreover, writing $B^*=C+R_0+2a^*\varphi^*T$, then we also assume that for all $t\leq T$ and $k\delta\leq T$,
\begin{equation}
\label{defB}
||\U\e(t)||\leq B^*, \qquad ||\bar{U}\e(t)||\leq b^{*} B^*, \qquad ||Y^{(\ep,\delta,\ell,E,\tau)}(k\delta)||\leq B^*,
\end{equation}
where $\bar{\U}\e(t)=\big(\bar{\U}\e_i(t),i\in\Lamepi\big).$
By Assumption \ref{ass:2}, \eqref{boundass3} and Proposition \ref{numDeDispEmDelta} such assumptions provide a small error probability.

In what follows, $C$ is a constant which may change from one appearance to another. 
We shall now proceed as follows. We shall first control the increments of $\cB_k$.  We next provide an upper bound for $\theta_k$ and lastly we conclude the proof.
\subsubsection*{Controlling the increments of $\cB_n$:}

We start noticing that
 $$ |\cB_k| \leq |\cB_{k-1}| + | A_k^1 \cap \mathcal G_{k-1} |  + | A^2_k \cap \mathcal G_{k-1} | \le |\cB_{ k- 1 }| + |A^1_{k}|+ | A^2_k \cap \mathcal G_{k-1}|, $$ where $\mathcal G_{k-1}$ is the set of good labels at time $k\delta$ (recall Definition \ref{goodlabels}) and
\begin{itemize}
\item $A^1_k$ is the set of all labels $i$ for which the clocks $\xi_i^1$ and $\xi_i$ associated to label $i$ ring during $ [ (k- 1) \delta , k \delta ],$
\item
$A^2_k$ is the set of all labels $i$ for which a clock $\xi_i^{2}$ associated to label $i$ rings during $ [ (k- 1) \delta , k \delta ]$. 
\end{itemize}
Recall the definitions of the random clocks $\xi^1_i,\xi_i^2$ and $\xi_i$ appearing in the coupling algorithm given in Subsection \ref{subsec5.1}.
Our aim is to prove that
\begin{eqnarray}
\label{M_n1}
&& P \Big[ |A^1_k | > \epsilon^{-2}  (\delta \varphi^* )^2 \Big] \le e^{ - C \epsilon^{-2} \delta^4} ,
\\&&
\label{M_n2}
  P \Big[  | A^2_k \cap \mathcal G_{k-1} | > 2 C \epsilon^{-2} \delta \left[  \theta_{k-1} +  \delta +\ell\right]  \Big]
  \le e^{ - C \epsilon^{-2} \delta^4 },
    \end{eqnarray}
where the constant $C$ appearing in $\eqref{M_n1}$ and $\eqref{M_n2}$ may be different. 

Then,  from \eqref{M_n1} and \eqref{M_n2}, we deduce that with probability $\geq 1 -  2e^{ - C \epsilon^{-2}  \delta^4 }$, 
    \begin{equation}
    \label{eq:429}
|\cB_k| \le |\cB_{k-1}| + \epsilon^{-2}(\delta \varphi^* )^2 + 2 C \epsilon^{-2} \delta
\left[  \theta_{k-1} +  \delta \right]\le |\cB_{k-1}| + C \epsilon^{-2} \delta
\left[ \theta_{k-1} + \delta \right].
    \end{equation}
Iterating the  above bound and using that $ k \le T\delta^{-1},$ we immediately get that
with probability $\geq 1 - 2k  e^{ - C \epsilon^{-2}  \delta^4 } \geq 1 - \delta^{-1}C  e^{ - C \epsilon^{-2}  \delta^4 },$ 
\begin{equation}
\label{BoundB_k}
\varepsilon^2|\cB_k| \le \varepsilon^2|\cB_1|+ C\delta \sum_{h=1}^{k-1} (\theta_{h} + \delta), 
\end{equation}
where $C$ depends only on $T$. Since
by definition $\theta_{k} \le \theta_{k+1}$, we may bound the right-hand of \eqref{BoundB_k} by $ C(\theta_{k-1}+\delta),$ implying that with probability $\geq 1 - \delta^{-1}C  e^{ - C \epsilon^{-2}  \delta^4 }, $
\begin{equation}
\label{boundB_k}
\varepsilon^2|\cB_k|\leq C(\theta_{k-1}+\delta),
\end{equation} 
for each $k\leq T\delta^{-1}.$

{\it Proof of \eqref{M_n1}}.
The random variable $ |A^1_k |$ is stochastically dominated by $Z^*:= \sum_{ i\in\Lamepi} \one_{ \{ Z_i^*  \geq 2 \} } ,$ where $ Z_1^*, \ldots Z_N^*  $ are independent Poisson variables of parameter $\varphi^*\delta$. Thus, writing $p^* = P ( N_i^*  \geq 2 ) $, we have
   $$
   e^{ - \delta \varphi^* } \frac12 \delta^2 (\varphi^*)^2 \le   p^* \le \frac12  (\delta \varphi^*)^2,\quad
   p^* \approx   \frac{ 1}{2 }\,(\delta \varphi^*)^2  \mbox{ as } \delta \to 0 .
   $$
Therefore, $Z^*$ is the sum of $\ep^{-2}$ Bernoulli random variables, each having mean value $p^*$. Invoking
the Hoeffding's inequality, we get  \eqref{M_n1}.

%

{\it Proof of \eqref{M_n2}.}
We shall dominate stochastically the random variable
$ |A^2_k \cap \mathcal G_{k-1} | $ 
by
\begin{equation}
\label{BoundA_k}
\bar{Z}:=\sum_{i\in \Lamepi} \mathbf \one_{\{\bar Z_i \ge 1\}},
\end{equation}
where $\bar Z_i, i\in\Lamepi , $ are independent Poisson variables of parameter $ C ( \theta_{k-1} + \delta +\ell) \delta.$ Once \eqref{BoundA_k} is established,  \eqref{M_n2} will follow straightly.

Noticing that, since,
   \[
|A_k^2 \cap \mathcal G_{k-1}|\leq \sum_{i\in \Lamepi}
\mathbf \one_{\{\xi^2_i<\delta, i\in \mathcal G_{k-1}\}}, 
   \]
it suffices to show the intensity of each random clock $\xi^2_i$, $i\in \mathcal{G}_{k-1},$ is $\leq C ( \theta_{k-1} + \delta +\ell) \delta.$   

For that sake, we shall write 
$$
y:=Y^{(\ep, \delta, \ell, E, \tau)}((k-1)\delta), \quad 
u:=\U\e((k-1)\delta) \quad \mbox{and} \quad
u_t:=\U\e((k-1)\delta+t), \ \  t\in[0,\delta).
$$
Now, for any $i\in \mathcal G_{k-1}\cap C_m$, the intensity of $\xi_i^2$ is
   \[
    |\varphi(u_i(t),i) - \varphi(y_i,i_m) | \le \|\varphi \|_{Lip} \big[|u_i(t)-y_i|+\ell\big],
    \]
where $\|\varphi\|_{Lip}$ is the Lipschitz constant of the function $\varphi.$  
Denoting the number of spikes of $U_j$ in interval $[s,t]$ by $N_j\big([s,t]\big)$, we have
\begin{multline}
|u_i(t)-y_i|\le 
|u_i-y_i|e^{-t(\al+\la_i)} +y_i\Big(1-e^{-(\alpha+\lambda_i)\delta}\Big) + 
\lambda_i\int_0^t \bar{u}_i(s)e^{-(\alpha+\lambda_i)(t-s)}ds \\ +\ep^2\sum_{j\in\Lamepi}a(j,i)N_j\big([(k-1)\delta,(k-1)\delta+t]\big).\hspace{1cm}
\end{multline}
Since for all $i\in\Lamepi$, $y_i,\bar{u}_i(s)\leq B^*$ and $\sum_{j\in\Lamepi}a(j,i)N_j\big([(k-1)\delta,(k-1)\delta+t]\big)\leq 2(a\varphi)^*\ep^{-2}\delta,$ then if additionally $i\in\mathcal{G}_{k-1}$, it follows that
$$|u_i(t)-y_i|\leq \theta_{k-1}+(\al+\la_i)\delta+\la_i\delta + 2a^*\varphi^*\delta,$$ 
and thus
 $$
|\varphi(u_i(t),i) - \varphi(y_i,i_m) | \le \|\varphi\|_{Lip}\left(  \theta_{n-1}
+ 2(\al+\sup_i\la_i)\delta +2a^*\varphi^*+ \ell \right) \le C ( \theta_{k-1} + \delta +\ell),
 $$
which implies that
$$|A_k\cap \mathcal{G}_{k-1}|
\le    \sum_{i\in\Lamepi}
\mathbf \one_{\{\bar{Z}_i<\delta\}} \quad \text{stochastically,}
$$
where the $\bar{Z}_i$ are independent Poisson random variables of intensity $C(\theta_{k-1} +\delta+\ell)\delta$.
\subsubsection*{Estimates on  $\theta_k$:}

Notice that ${\cal G}_{k}={\cal G}_{k- 1} \cap (C_k \cup F_k )$ where:

\begin{enumerate}
\item $C_k$ is the set of all indexes $i$ whose associated  random clock  $\xi_i^{1}$ rings only once during $ [ (k- 1) \delta , k \delta ]$.
\item
$F_k$ is the set of indexes $i$ which did not spike during $ [ (k- 1) \delta , k \delta ] .$
\end{enumerate}
In what follows, we will make use  of the expression for membrane potential
$\U\e_i(t)$ of a neuron which did not spike in the interval $[s,t]$:
\begin{equation}
\label{formule}
\U\e_i(t) = e^{-(\al+\la_i)(t-s)}\U\e_i(s) +\la_i\int_{s}^t e^{-(\al+\la_i)(t-h)}\left\{\bar{\U}\e_i(h)dh + \frac{\ep^2}{ \la_i}\sum_{j\in\Lamepi}a(j,i)dN_j(h)\right\},
   \end{equation}
being $N_j(t)$ the total number of spike in the process $\U$ of neuron $j$  up time $t$.

\medskip
$\bullet$\;\;  Take $i\in C_k \cap {\cal G}_{k- 1}$.
In this case, we have that for some time $ t \in [(k-1)\delta,k\delta[$, the random clock 
$\xi_i^{1 }=t$. By \eqref{formule},
    $$
\U\e_i(k \delta)  = \la_i\int_t^{ \delta }
e^{ - (\al+\la_i) ( \delta  - s)} \bar \U\e_i (s) ds +
e^{ - (\al+\la_i)  \delta  }\ep^2\sum_{j\in\Lamepi} \int_t^\delta e^{(\al+\la_i )s} d   N_j(s)   ,
   $$
since $U\e_i(t_+)=0$. Noticing also that $||\bar{\U}\e (t)|| \le B^*$
and $N\big([(k-1)\delta,k\delta)\big) \le 2{\varphi}^*\delta \ep^{-2}$ we immediately see that
$U\e_i (k \delta ) \le C \delta$. By similar arguments, $ Y\Prm_i (k \delta )  \le C \delta ,$ 
so that
\begin{equation}
\label{eq:din2}
 D_i (k  ) \le  C \delta .
\end{equation}
Observe that the value $D_i(k-1)$ does not appear on the bound above. We shall now analyse the other case.

\medskip
$\bullet$\;\; Fix $i\in F_k \cap {\cal G}_{k- 1}$. Notice that the neuron $i$ is good at time $ (k-1) \delta $ and did not
spike in the time interval $[(k-1)\delta, k\delta)$ neither in the $U\e$ nor in the $Y\Prm$ processes.
As before, we write $ \U\e((k-1) \delta)= u $ and $ Y\Prm((k-1) \delta)= y.$
By \eqref{formule} and \eqref{flowy}, the variable  $|\U_i ( k \delta ) - Y\d_i ( k \delta ) | = D_i(k)$, $i\in C_m$, is bounded by
    \begin{multline}
    \label{eq:538}
D_i (k) \le  \left|e^{-\delta(\al+\la_i)}u_i - e^{-\delta(\al+\la_m)}y_i \right| \\ +\left|
 \int_{(k-1)\delta}^{k\delta} \la_i e^{- (\al+\la_i) (k\delta -t )}\bar \U\e_i (t)  dt - \la_{m}\int_{(k-1)\delta}^{k\delta} \bar y(m) e^{- (\al+\la_{m}) (k\delta -t )}  dt \right|\\
+ \left|\ep^2\sum_{j\in \Lamepi}a(j,i)\int_{(k-1)\delta}^{k\delta}  e^{-(\al+\la_j)(k\delta -t)}  dN_j(t)  - \ep^2\sum_{m'}a(i_{m'},i_m)\tilde{N}\big([(k-1)\delta, k\delta)\big) \right|,
\end{multline}
where $\tilde{N}\big([(k-1)\delta,k\delta)\big)$ denote the number of spikes of the $Y^{(\ep,\delta,\ell,E,\tau)}$ process in the interval $[(k-1)\delta,k\delta).$
Thus, it suffices to bound each term on the right hand side of \eqref{eq:538}. 

We start bounding the first one: 
\begin{eqnarray*}
|e^{-\delta(\al+\la_i)}u_i - e^{-\delta(\al+\la_{m})}y_i |\leq B^*\delta |\la_i-\la_{m}|+ e^{-(\al+\la_{m})\delta}|u_i-y_i|.
\end{eqnarray*}
Since, $|\la_i-\la_{m}|\leq ||\la||_{\rm Lip}\ell$, and supposing $\ell\leq \delta$, we can bound the last sum by $C\delta^2+\theta_{k-1}.$

Now let's bound the second term on the right-hand side of \eqref{eq:538}. It is easy to see that it is bounded by
\begin{multline*}
||\la||_{\rm Lip}B^*\ell\delta(1+\la_{m})+\la_{m}\delta |\bar y(m)-\bar u_{i_m}|+\la_{m}\int_{(k-1)\delta}^{k\delta}\Big[\big|\bar \U_i(t)-\bar u_i\big|+\big|\bar \U_{i_m}(t)-\bar u_{i_m}\big|\Big]dt .
\end{multline*}
To control the second and third terms we notice that for any $i\in\Lamepi$, $|\bar U_i(t)-\bar{u}_i|\leq C\delta$ and $|\bar U_i(t)-\bar y_i|\leq C\delta$. In addition, for any $i\in C_m$, $m=1,\ldots, \ell^2$, $|\bar U_i(t)-\bar u_{i_m}|\leq C\ell$. Requiring that $\ell\leq\delta$,  from these three inequalities we can bound the sum above by $C\delta(\delta +\theta_{k-1}).$ 

The argument to bound the third term on \eqref{eq:538} is a bit more tricky.  First we bound that term by
\begin{multline*}
\ep^2\sum_{j}a(j,i)\int_{(k-1)\delta}^{k\delta}(k\delta-t)(\al+\la_j)dN_j(t)+ \ep^2\sum_{m'}\sum_{j\in C_{m'}}\big|a(j,i)-a(i_{m'},i_m)\big|N_j\big([\delta(k-1),k\delta)\big)\\
+\ep^2\sum_{m'}a(i_{m'},i_m)\Big|N_{C_{m'}}\big([(k-1)\delta,k\delta)\big)-\tilde{N}_{C_{m'}}\big([(k-1)\delta,k\delta)\big)\Big|,
\end{multline*}
where $N_{C_{m'}}\big([(k-1)\delta,k\delta)\big)=\sum_{j\in C_{m'}}N_j\big([(k-1)\delta,k\delta)\big)$ is total number of spikes in the $\U\e$ process inside the square $C_{m'}$ during the time interval $[(k-1)\delta,k\delta))$ and $N_{C_{m'}}\big([(k-1)\delta,k\delta)\big)$ is the correspondent quantity associated to the $Y^{(\ep,\delta,\ell,E,\tau)}$ process. 

The first two terms above are easily bounded. One can check that the sum of the two can be bounded by $C\delta^2$. To control the third term,  we shall show that 
\begin{equation}
\Big|N_{C_{m'}}\big([(k-1)\delta,k\delta)\big)-\tilde{N}_{C_{m'}}\big([(k-1)\delta,k\delta)\big)\Big|\leq 4 (\varphi^*\delta)^2 \ep^{-2}\ell^2
\end{equation}
Indeed, its difference is smaller or equal to 
\begin{equation}
\label{formula9.13}
\sum\limits_{j\in C_{m'}\cap A^1_k}N_j([(k-1)\delta,k\delta)) + |C_{m'}\cap A^2_k |,
\end{equation}
so that it suffices to control this two terms. We star with the second one. We know that with probability $\geq 1-e^{ - C \epsilon^{-2} \delta^4 }$,
$$
|C_{m'}\cap A^2_k |=|C_{m'}\cap A^2_k \cap \mathcal{G}_{k-1}|+ |C_{m'}\cap A^2_k \cap \mathcal{B}_{k-1}|\leq  2Cl^2\ep^ {-2}\delta(\theta_{k-1}\delta)+ C\delta \ell^2 |\mathcal{B}_{k-1}|,
$$
where we used \eqref{M_n2} and that the number of neurons in $\mathcal{B}_{k-1}\cap C_{m'}$ which spiked in a time $\delta$ is dominated by a Poison random variable of rate $\varphi^*\delta |\mathcal{B}_{k-1}\cap C_{m'}|.$ Thus, it remains only to bound the first term in \eqref{formula9.13}. 

In order to do that, we start noticing that
\begin{multline}
   \label{eq:539c}
P\Big[ \sum_{j\in A^1_{k}\cap C_{m'}} N_j((k-1)\delta,k\delta) \ge 4 (\varphi^*\delta)^2 \ep^{-2}\ell^2\Big]
\le P\Big[ \sum_{j\in A^1_{k}\cap C_{m'}} N_j((n-1)\delta,n\delta)\\ \ge 4 (\varphi^*\delta)^2 \ep^{-2}\ell^2; |A^1_{k}\cap C_{m'}| \le
(\varphi^*\delta)^2 \ep^{-2}\ell^2 \Big] + P\Big[ |A^1_{k}\cap C_{m'}| >
(\varphi^*\delta)^2 \ep^{-2}\ell^2 \Big] .
   \end{multline}
The second term is controlled by the estimate on \eqref{M_n1}.  Let $A \subset C_{m'} $, $|A| \le (\varphi^*\delta)^2 \ep^{-2}\ell^2 ,$ then
   \[
P\Big[ \sum_{j\in A^1_{k}\cap C_{m'}} N_j((k-1)\delta,k\delta)\\ \ge 4 (\varphi^*\delta)^2 \ep^{-2}\ell^2\;|\; A^1_{k}\cap C_{m'} =A  \Big]
\le P^*\Big[ \sum_{j\in A}( N^*_j-2) \ge  2 (\varphi^*\delta)^2 \ep^{-2}\ell^2 \Big] ,
   \]
being $P^*$ the distribution of independent Poison random variables $N^*_j$, $j\in A$, each having parameter $\varphi^*\delta$
and
conditioned on being $N^*_j\ge 2$.  In this way, we easily get that 
   \[
 P^*[N^*_j-2 = k ]=  Z_\xi^{-1} \frac{\xi^k}{(k+2)!},\quad Z_\xi = \xi^{-2} \Big(e^\xi - 1 -\xi\Big),\quad
 \xi = \varphi^*\delta .
   \]
No let $X_1,X_2, \ldots, $ be a sequence of independent Poison variables with parameter $\xi$. It follows that
$N^*_j-2 \le X_j$ stochastically for $\xi$ small enough, hence
for $\delta$ small enough.  Indeed
for any integer $k$ we have
 \begin{equation}
   \label{eq:540a}
 P^*[N^*_j-2 \ge k ] \le P[ X_j \ge k]
   \end{equation}
because for $k\ge 1$,
   \[
P^*[N^*_j-2 \ge k ]  \le  \frac{2\xi^k}{(k+2)!},\quad
P[ X_j \ge k] \ge e^{-\xi} \frac{\xi^k}{k!} ,
   \]
hence \eqref{eq:540a} when $3e^{-\xi} \ge 2$.

Since
$X=\sum_{j\in A} X_j$ is a Poisson variable of parameter $|A| \xi \le (\varphi^*\delta)^2 \ep^{-2}\ell^2 \varphi^*\delta$
we have
   \[
 P^*\Big[ \sum_{j\in A}( N^*_j-2) \ge  2 (\varphi^*\delta)^2 \ep^{-2}\ell^2  \Big] \le P^*\Big[X \ge  2 (\varphi^*\delta)^2 \ep^{-2}\ell^2  \Big] ,
   \]
where the expectation $E^* (X) $ of $X$ is smaller (for $\delta$ small) than $(\varphi^*\delta)^2 \ep^{-2}\ell^2$.
As a consequence,    
\[
P^*\Big[ \sum_{j\in A}( N^*_j-2) \ge  2 (\varphi^*\delta)^2 \ep^{-2}\ell^2  \Big] \le e^{-C\epsilon^{-2} \delta^{2}\ell^2} .
   \]
To sum up, we have for $i\in F_k \cap {\cal G}_{k- 1}$ with probability $\ge 1- e^{-C\epsilon^{-2} \delta^{2}\ell^2}$,
    \begin{equation}
    \label{eq:538ter}
D_i (k) \le  \theta_{k-1}(1+ C \delta ) + C\delta  |\mathcal{B}_{k-1}|\ep^2+ C\delta^2.
  \end{equation}
The above inequality together with \eqref{eq:din2} guarantee
that with probability $\ge 1- e^{-C\epsilon^{-2} \delta^{2}\ell^2},$
    \begin{equation}
    \label{eq:538fourth}
\theta_k \le \max\{ C\delta; \theta_{k-1}(1+ C \delta ) + C\delta  |\mathcal{B}_{k-1}|\ep^2+ C\delta^2\} .
  \end{equation}

\vskip.5cm

\subsubsection*{Iteration on the bound of $\theta_k$:}
As a consequence of \eqref{boundB_k}, $ \ep^{2}|\mathcal{B}_k| \le C (\theta_{k-1} + \delta )$ for all $k\delta \le T$
with probability $ 1 - \delta^{-1}C  e^{ - C \epsilon^{-2}  \delta^4 }$.  As a by product of \eqref{eq:538fourth},
with probability  $ 1 - \delta^{-1}C  e^{ - C \epsilon^{-2}  \delta^4 }$, it follows that
     $$
\theta_k \le \max \Big( C \delta ,\left[ 1 + C \delta \right] \theta_{k-1} + C \delta^2 \Big) .
     $$
As a direct consequence (iterate the above inequality), it holds
\begin{equation*}
 \theta_k \le   C \sum_{ s=0}^{k-1} \left[ 1 + C \delta \right]^s  \delta^2 + (1 + C \delta)^k  C \delta,
\end{equation*}
and since, 
\begin{equation*}
 C \sum_{ s=0}^{k-1} \left[ 1 + C \delta \right]^s  \delta^2 + (1 + C \delta)^k  C \delta=C\delta [\left( 1 + C \delta \right)^k - 1]   + (1 + C \delta)^k  C \delta \\
\le  C e^{ C T } \delta\ ,
\end{equation*}
remember that $ k \delta \le T,$
we conclude that 
$$ \theta_k \le C \delta $$
for all $ \delta \le \delta_0, $ with probability $\geq 1 - \delta^{-1}C  e^{ - C \epsilon^{-2}  \delta^4 }$. This finishes the proof of Theorem \ref{thm:4}.

\section{Appendix 2: proof of proposition \ref{prop3} }
\label{proof prop3}
\begin{proof}
Fix $\phi\in \mathcal{S}$. By  (\ref{defB}), the left-hand side of \eqref{eq:445'} does not change if we consider $U^*(t)=\min\{\U\e(t),B^{*}\}$ and $Y^*(t)=\min\{Y^{(\ep,\delta,\ell,E,\tau)}(t),B^*\}$ instead of $\U\e(t)$ and $Y^{(\ep,\delta,\ell,E,\tau)}(t)$. 
Now, by the smoothness of the function $\phi$,  

$$Q_u^{(\ep)}\left[\Big|\ep^2\sum_{i\in C_m}\phi(U_i(t),i)-\ep^2\sum_{i\in C_m}\phi(Y_i(t),i_m)\Big|\right]\leq ||\varphi||_{\rm Lip}Q_u^{(\ep)}\Big[\ep^2\sum_{i\in C_m} |U^*(t)-Y^*(t)|\Big].$$
Applying the Theorem \ref{thm:4} and using that $|U^*(t)-Y^*(t)|\leq B^*$,  we get the desired upper bound in \eqref{eq:445'}.
\end{proof}    

\section{Appendix 3: proof of Theorem \ref{thm:5}}
\label{proof Thm5}
\begin{proof}

Let $\mathcal{F}_n$ be the sigma-algebra generated by the variables $\xi_i=\xi_i(k), k\leq n-1,i\in\Lamepi$ appearing in  \eqref{numberspikes}. Observe that all variables $Y^{(\ep,\delta,\ell,E,\tau)}(n\delta)$, $e\e_{n}(m)$, $S\e_{n+1}(m,h)$ and $S\e_{n+1}(m,\delta)$ are $\mathcal{F}_{n}-$ measurable.
In what follows, the constants $C,c_1$ and $c_2$ may change from appearance to another.  We also will write for simplicity $\tilde{E}\e=\tilde{E}_{Y^{(\ep,\delta,\ell,E,\tau)}(0)}^{(\ep)}$.

The proof is made by induction. For $n=0$, the proposition is easy to check. Indeed, notice that in this case $E\e_{0,k}=D\e_{0,k}$ . Moreover,  notice also that 
$$\zeta_{0,m}(D\e_{0,k})=\tilde{E}^{(\ep)}\big[\eta_{0,m}(E\e_{0,k})\big]=\sum_{i\in C_m}\int_{I_k}\psi_0(u,i)du $$
and that $\eta_{0,m}(E\e_{0,k})$ is a sum of $\ell^{2}\ep^{-2}$ independent Bernoulli random variables $X_i, i\in C_m$, where expected value of $X_i$ is $\int_{I_k}\psi_0(u,i)du.$
By Hoeffding inequality we deduce that
$$ \varepsilon^2 |\eta\e_{0,m}(E_{0,k}) -\zeta\e_{0,m}(D_{0,k})|>E\ell^2\varepsilon^{1/2}=C$$
with probability $\leq 2e^{-c_2\varepsilon^{-1}}$ where $c_2=2E^2\ell^2.$  
Therefore, 
it follows, for $n=0$, that the inequality above holds for all $k$ and $m$ with probability larger or equal to
\begin{equation*}
\label{probn=0}
1-c_1 e^{-c_2\varepsilon^{-1}},
\end{equation*}
establishing the Theorem in the case $n=0$. We now suppose that the result holds for $k\leq n$. Introduce  the set $G_n$ in which: 
\begin{itemize}
\item $\big| E\e_{n,k}-D\e_{n,k} \big| \le C\varepsilon^{1/2},$ $k=1,\ldots, |\mathcal{E}\e_{n}|$ 
\item $\varepsilon^2 \Big|\eta_{n,m}\Big(E\e_{n,k+\delta\tau^{-1}}\Big) -\zeta_{n,m}\Big(D\e_{n,k+\delta\tau^{-1}}\Big)\Big| \le E\ell^2\varepsilon^{1/2}, $ $k=1,\ldots, |\mathcal{E}\e_{n}|,$ and 
\item $\varepsilon^2 \Big|\eta_{n,m}\Big(E\e_{n,h}\Big) -\zeta_{n,m}\Big(D\e_{n,h}\Big)\Big| \le \tau\ell^2\varepsilon^{1/2}, \ h=1,\ldots, \delta\tau^{-1}.$ 
\end{itemize}
By the inductive hypothesis, $\tilde{P}\e(G_n)\geq 1-c_1e^{-c_2\ep^{-1/2}}.$

Since, 
\begin{multline*}
|E\e_{n+1,k+\delta\tau^{-1}}-D\e_{n+1,k+\delta\tau^{-1}}| \le  |E\e_{n,k}-D\e_{n,k}| + \lambda_m \delta \Big|\bar{y}\e_{n}(m)-e\e_{n}(m)\Big|\\ + \Big|S\e_{n+1}(m,\delta)-\tilde{E}^{(\ep)}\big[S\e_{n+1}(m,\delta)\big]\Big|,
\end{multline*}
we have that  on $G_n,$ 
\begin{multline*}
|E\e_{n+1,k+\delta\tau^{-1}}-D\e_{n+1,k+\delta\tau^{-1}}| \le C\ep^{1/2}+\Big|S\e_{n+1}(m,\delta)-\tilde{E}^{(\ep)}\big[S\e_{n+1}(m,\delta)\big]\Big|.
\end{multline*}
We shall show that there exist positive constants $c,c_1$ and $c_2$ not depending on $\ep$ such that
\begin{equation}
\label{diferenceonspikes}
\Big|S\e_{n+1}(m,\delta)-\tilde{E}^{(\ep)}\big[S\e_{n+1}(m,\delta)\big]\Big|\leq c\ep^{1/2},
\end{equation}
with probability $\geq 1-c_1e^{-c_2\ep^{-1}}$. For that sake, we first write 
$$N_{n+1}(m,k,\delta)=\sum\limits_{i\in C_m}\one_{\{\xi_i< \delta \}}, \ \xi_i \sim \exp(\varphi (E\e_{n,k},i_m) $$
and then by the conditional  version of Hoeffding's  inequality we deduce that 
\begin{equation}
\label{probcond1}
\tilde{P}^{(\ep)}(\ep^{2}\big| N_{n+1}(m,k,\delta)- \eta_{n,m}(E\e_{n,k})(1-e^{-\delta\varphi(E\e_{n,k},i_m)})\big|> E\ell^2 \ep^{1/2}|\mathcal{F}_n)\leq c_1e^{-c_2\varepsilon^{-1}}
\end{equation}
Since on $G_n$ 
$$|\varphi(E\e_{n,k},i_m)-\varphi(D\e_{n,k},i_m)|\leq C\ep^{1/2},$$
noticing that $N_{n+1}(m,\delta)=\sum_{k}N_{n+1}(m,k,\delta)$ and $\ep^2\zeta(D\e_{n,k})\leq 1,$ then it follows together with \eqref{probcond1} that there exist constants $C,c_1$ and $c_2$ such that
$$\tilde{P}^{(\ep)}\Big(G_n,\ep^{2}\big|S\e_{n+1}(m,\delta)-\tilde{E}\e\big[S\e_{n+1}(m,\delta)\big]\big|>C\ep^{1/2}\Big|\mathcal{F}_n\Big)\leq c_1e^{-c_2\ep^{-1}},$$
proving \eqref{diferenceonspikes}. 
Therefore,
$$\tilde{P}^{(\ep)}\Big(G_n,|E\e_{n+1,k+\delta\tau^{-1}}-D\e_{n+1,k+\delta\tau^{-1}}|>C\ep^{1/2}\Big|\mathcal{F}_n\Big)\leq c_1e^{-c_2\ep^{-1}}.$$
A similar argument may be used  to prove that we may replace in the probability above $E\e_{n+1,k+\delta\tau^{-1}}$ and $D\e_{n+1,k+\delta\tau^{-1}}$ respectively  by $E\e_{n+1,h}$ and $D\e_{n+1,h}$. Thus, summing over all $k$,$h$ and $m$ we prove the first part of Theorem \ref{thm:5} for $n+1.$

Now, we noticing that $\eta_{n+1}(m,k+\delta\tau^{-1})=\eta_{n}(m,k)-N_{n+1}(m,k,\delta)$ and remembering that by \eqref{numdet1}, $\zeta_{n+1}(m,k+\delta\tau^{-1})=\zeta_{n+1}(m,k)e^{-\delta\varphi\big(D\e_{n,k},i_m\big)},$ we easily see, together with \eqref{probcond1}, that
$$\tilde{P}^{(\ep)}\Big(G_n,\varepsilon^2 \big|\eta\e_{n+1}(m,k+\delta\tau^{-1}) -\zeta_{n+1}(m,k+\delta\tau^{-1})\big|>C\ep^{1/2}\Big|\mathcal{F}_n\Big)\leq c_1e^{-c_2\ep^{-1}},$$
for some suitable constants not depending on $\ep$. A similar argument shows that the same type of bound for $\varepsilon^2 \big|\eta\e_{n+1}(m,h) -\zeta\e_{n+1}(m,h)\big|$ also holds, finishing the proof of Theorem \ref{thm:5}. 

\end{proof}

\section{Appendix 4: proof of Theorem 2 for general firing rates}
\label{generalcase}

The proof is analogous to the proof presented in Appendix 4 of \cite{Errico:14}. For sake of completeness we shall give it here.

Let $\varphi, R, T$ and $C$ as in the statement of Theorem \ref{thm:1} and take $\phi$ be any bounded continuous functions on $D\big([0,T], \mathcal{S}' \big).$ We have to show that 
$$\lim_{\ep\to 0} \mathcal{P}^{(\ep)}_{[0,T]}(\phi)=\phi(\rho).$$

Let $A$ be the set $A=\{||U\e(t)||\leq C,t\in [0,T]\}.$ Theorem \ref{thm:1} implies that
\begin{equation}
\label{gen1}
\lim_{\ep\to 0 } \big| \mathcal{P}^{(\ep)}_{[0,T]}(\phi) - \mathcal{P}^{(\ep)}_{[0,T]}(\phi1_A)\big|=0.
\end{equation} 
Now, consider $\mathcal{P}^{(*,\ep)}_{[0,T]}$ the distribution of the process with a spiking rate $\varphi^*(\cdot,\cdot)$ which fulfils the Assumption \ref{ass:3} and it is equal to $\varphi$ for $u\leq C$. By definition, it follows that 
\begin{equation}
\label{gen2}
\mathcal{P}^{(\ep)}_{[0,T]}(\phi1_A)=\mathcal{P}^{(*,\ep)}_{[0,T]}(\phi1_A).
\end{equation}
Having proved Theorem \ref{thm:2} under the Assumption \ref{ass:3}, we get the desired convergence to a limit density $\rho^*=(\rho^*_tdudr)_{t\in[0,T]}$, for the process whose spiking rate is $\varphi^*$. 
It follows then, from \eqref{gen1} and \eqref{gen2}, that
$$\lim_{\ep\to 0 }\mathcal{P}^{(\ep)}_{[0,T]}(\phi)=\psi(\rho^*1_A).$$
We claim that $\rho^*=\rho^*1_A$. Indeed, by considering $\phi(w)=\sup\{w_t(1),t\leq T\} \wedge1,$
we immediately see that $1=\lim_{\ep\to 0 } \mathcal{P}^{(\ep)}_{[0,T]}(\phi)=\phi(\rho^*1_A).$ This last equalty implies that $\rho^*$ have support in $[0,C].$ As a consequence, 
$$\lim_{\ep\to 0 } \mathcal{P}^{(\ep)}_{[0,T]}(\phi)=\phi(\rho^*1_A)=\phi(\rho^*),$$
which concludes the proof of the Theorem. 
\bibliography{Bibli}{}

\begin{thebibliography}{10}

\bibitem{billing99}
Patrick Billingsley.
\newblock {\em Convergence of probability measures}.
\newblock Wiley Series in Probability and Statistics: Probability and
  Statistics. John Wiley \& Sons Inc., second edition, 1999.
\newblock A Wiley-Interscience Publication.

\bibitem{Davis:84}
M.~H.~A. Davis.
\newblock Piecewise-deterministic {M}arkov processes: a general class of
  nondiffusion stochastic models.
\newblock {\em J. Roy. Statist. Soc. Ser. B}, 46(3):353--388, 1984.

\bibitem{Errico:14}
A.~De~Masi, A.~Galves, E.~Löcherbach, and E.~Presutti.
\newblock Hydrodynamic limit for interacting neurons.
\newblock {\em Journal of Statistical Physics}, 158(4):866--902, 2015.

\bibitem{ANAERRICO}
A~De~Masi and E.~Presutti.
\newblock {\em Mathematical methods for hydrodynamic limits}.
\newblock Lecture notes in mathematics. Springer-Verlag, 1991.

\bibitem{AG:14}
A.~Duarte and G.~Ost.
\newblock A model for neural activity in the absence of external stimulus.
\newblock {\em ArXiv}, 2014.

\bibitem{Evafou:14}
N.~Fournier and E.~L\"ocherbach.
\newblock On a toy model of interacting neurons.
\newblock {\em ArXiv}, 2014.

\bibitem{GalEva:13}
A.~Galves and E.~L\"ocherbach.
\newblock Infinite systems of interacting chains with memory of variable
  length—a stochastic model for biological neural nets.
\newblock {\em Journal of Statistical Physics}, 151(5):896--921, 2013.

\bibitem{EvaGalves:15}
A.~Galves and E.~L\"ocherbach.
\newblock Modeling networks of spiking neurons as interacting proces ses with
  memory of variable length.
\newblock {\em ArXiv}, 2015.

\bibitem{Gerstner:2002:SNM:583784}
Wulfram Gerstner and Werner Kistler.
\newblock {\em Spiking Neuron Models: An Introduction}.
\newblock Cambridge University Press, New York, NY, USA, 2002.

\bibitem{Landim}
C.~Kipnis and C.~Landim.
\newblock {\em Scaling limits of interacting particle systems}.
\newblock Grundlehren der mathematischen Wissenschaften. Springer, Berlin, New
  York, 1999.

\bibitem{Thieullen:12}
M.~Thieullen M.~Riedler and G.~Wainrib.
\newblock Limit theorems for infinite-dimensional piecewise deterministic
  markov processes. applications to stochastic excitable membrane models.
\newblock {\em Electron. J. Probab.}, 17:no. 55, 1--48, 2012.

\bibitem{mitoma:1983}
I.~Mitoma.
\newblock Tightness of probabilities on $c(\lbrack 0, 1 \rbrack; \mathscr{Y}')$
  and $d(\lbrack 0, 1 \rbrack; \mathscr{Y}')$.
\newblock {\em Ann. Probab.}, 11(4):989--999, 11 1983.

\bibitem{Presutti:08}
Errico Presutti.
\newblock {\em {Scaling limits in statistical mechanics and microstructures in
  continuum mechanics}}.
\newblock Theoretical and Mathematical Physics. Springer, Dordrecht, 2008.

\bibitem{Robert:14}
P.~Robert and J.~Touboul.
\newblock On the dynamics of random neuronal networks.
\newblock {\em ArXiv}, 2014.

\bibitem{touboul2014}
Jonathan Touboul.
\newblock Propagation of chaos in neural fields.
\newblock {\em Ann. Appl. Probab.}, 24(3):1298--1328, 06 2014.

\end{thebibliography}
\bibliographystyle{plain}
\nocite{GalEva:13}
\nocite{Errico:14}
\nocite{Davis:84}
\nocite{Thieullen:12}
\nocite{Gerstner:2002:SNM:583784}
\nocite{Robert:14}
\nocite{AG:14}
\nocite{Evafou:14}
\nocite{ANAERRICO}
\nocite{mitoma}
\nocite{EvaGalves:15}
\nocite{mitoma:1983}
\nocite{billing}
\nocite{touboul2014}
\nocite{billing99}
\nocite{Presutti:08}
\nocite{Landim}

\end{document}